\documentclass{amsart}
\usepackage{amsmath,amssymb,dsfont,amsthm}
\usepackage{graphics}
\usepackage{xypic}
\usepackage{latexsym}
\usepackage{amsmath}
\usepackage{amsfonts}
\usepackage{amssymb}
\usepackage{tensor}

\date{\today}

\title[Reidemeister torsion and analytic torsion of discs]{Reidemeister torsion and analytic torsion of discs}
\thanks{2000 {\em Mathematics Subject Classification: 57Q10 58J52}.\\
}

\author{T. de Melo, L. Hartmann  and M. Spreafico}
\address{\tt ICMC, Universidade S\~{a}o Paulo, S\~{a}o Carlos, Brazil.}
\email{mauros@icmc.usp.br}

\newtheorem{theo}{Theorem}
\newtheorem{lem}{Lemma}
\newtheorem{corol}{Corollary}
\newtheorem{defi}{Definition}

\newtheorem{prop}{Proposition}
\newtheorem{rem}{Remark}

\renewcommand{\Re}{{\rm Re}}
\renewcommand{\Im}{{\rm Im}}

\newcommand{\Sp}{{\rm Sp}}
\newcommand{\beq}{\begin{equation}}
\newcommand{\eeq}{\end{equation}}
\newcommand{\N}{{\mathds{N}}}
\newcommand{\Z}{{\mathds{Z}}}
\newcommand{\R}{{\mathds{R}}}
\newcommand{\C}{{\mathds{C}}}

\newcommand{\T}{{\mathcal{T}}}

\newcommand\e{{\rm e}}
\newcommand{\A}{{\mathcal{A}}}
\renewcommand{\H}{{\mathcal{H}}}

\renewcommand{\P}{{\mathcal{P}}}
\newcommand{\B}{{\mathcal{B}}}

\renewcommand{\b}{{\partial}}
\newcommand{\Vol}{{\rm Vol}}
\newcommand{\ec}{{\mathsf e}}
\renewcommand{\ge}{{\mathsf g}}

\renewcommand{\det}{{\rm det}}

\newcommand{\D}{{\partial}}

\date{}

\DeclareMathOperator*{\Rz}{Res_0}

\DeclareMathOperator*{\Ru}{Res_1}

\begin{document}


\maketitle

\begin{abstract} We study the  Reidemeister torsion and the analytic torsion of the $m$ dimensional disc in the Euclidean $m$ dimensional space, using the base for the homology defined by Ray and Singer in \cite{RS}. We prove that the  Reidemeister torsion coincides with a power of the volume of the disc. We study the additional terms arising in the analytic torsion due to the boundary, using generalizations of the Cheeger-M\"{u}ller theorem. We use a formula proved by Br\"uning and Ma \cite{BM}, that predicts a new anomaly boundary term beside the known term proportional to  the Euler characteristic of the boundary \cite{Luc}. Some of our results extend to the case of the cone over a sphere, in particular we evaluate directly the analytic torsion for a cone over the circle and over the two sphere. We compare the results obtained in the low dimensional cases. We also consider a different formula for the boundary term given by Dai and Fang \cite{DF}, and we show that the result obtained using this formula is inconsistent with the direct calculation of the analytic torsion.
\end{abstract}

\section{Introduction}
\label{s0}

The Reidemeister (R) torsion is an important topological invariant introduced originally by Reidemeister \cite{Rei}, Franz \cite{Fra} and de Rham \cite{deR} to classify lens spaces. For non acyclic space, the R torsion depends on the homology. However, dealing with Riemannian manifolds, Ray and Singer \cite{RS} introduce a geometric torsion invariant that we call Ray and Singer (RS) torsion. They use the Riemannian structure  to fixing the dependence on the homology of the R torsion. In the same work, in searching for an analytic description of the RS torsion, Ray and Singer also introduced the analytic torsion, that soon became an important geometric invariant on its own, and has been deeply investigated  by various authors (see for example \cite{BZ} and the references therein).
The equivalence between  the RS torsion and the analytic torsion, conjectured by Ray and Singer, was eventually proved by Cheeger \cite{Che1} and M\"{u}ller \cite{Mul}, for closed manifold. Cheeger also discussed the case of manifolds with boundary, showing that in this case an extra term could appear. Much lather, this boundary term was explicitly given by Lott and Rothenberg \cite{LR} and L\"{u}ck \cite{Luc}, for the case of manifolds with a product metric structure near the boundary. Only in 2000, Dai and Fang \cite{DF} gave a formula for the difference of the RS torsion and the analytic torsion on a manifold with boundary without any assumption for the metric near the boundary. In this formula some new terms appear. However, in a recent work of Br\"{u}ning and Ma  \cite{BM} on Ray-Singer metrics on manifolds with boundary, a further formula is given, where a different boundary contribution appears. The results given in Theorem \ref{t2} below are obtained using the formula of Br\"{u}ning and Ma. The results obtained using the formula of Dai and Fang are given at the end of Section \ref{s3}. 
Beside the intensive investigation and the large literature available, comparably few results exist on the quantitative side, namely explicit evaluations of the analytic torsion \cite{Ray} \cite{WY} \cite{dMS}. Continuing along this line of investigation, we study in this work the simplest case of a manifold with boundary, namely the case of a disc. Let $(W,g)$ be a compact connected Riemannian manifold with boundary $\b W$, and  metric $g$, and $\rho:\pi_1(W)\to O(k,\R)$ an orthogonal representation of the fundamental group of $W$. We denote by $\tau_{\rm RS}((W,g);\rho)$ the RS torsion, by $\tau_{\rm RS}((W,\b W,g);\rho)$ the RS torsion of the pair $(W,\b W)$. We denote by $T_{\rm abs}((W,g);\rho)$ the analytic torsion of $(W,g)$ with absolute boundary conditions on $\b W$, and by $T_{\rm rel}((W,g);\rho)$, the analytic torsion of $(W,g)$ with relative boundary condition, both with respect to the representation $\rho$ (see Section \ref{s1} for the precise definitions). Let $D^m_l=\{x\in \R^m~|~|x|\leq l\}$, the disc of radius $l>0$ in the euclidian space $\R^m$, and with  the standard metric $g_E$ induced by the immersion. With this notation, we now state our main results.

\begin{theo}\label{t1} The RS torsion of the disc $D^m_l$ of radius $l>0$ in $\R^m$ with the standard metric $g_E$ induced by the immersion in the Euclidean space, and an orthogonal  representation $\rho$ of the fundamental group, is:
\begin{align*}
\tau_{\rm RS}((D^m_l,g_E);\rho)=\left(\sqrt{{\rm Vol}_{g_E}(D^m_l)}\right)^{{\rm rk}(\rho)}.
\end{align*}

In the same situation, the RS torsion of the pair  $(D^m_l, S^{m-1}_l)$  is:
\begin{align*}
\tau_{\rm RS}((D^m_l, S^{m-1},g_E);\rho)&=\left(\sqrt{{\rm Vol}_{g_E}(D^m_l)}\right)^{(-1)^{m-1}{\rm rk}(\rho)}.
\end{align*}

\end{theo}

\begin{proof} The results follow from Propositions \ref{px} and \ref{p2} of Section \ref{s2}, taking $\alpha=\frac{\pi}{2}$.
\end{proof}

\begin{theo}\label{t2} The analytic torsion of the disc $D^{m}_l$ of radius $l>0$ in $\R^{m}$ with the standard metric $g_E$ induced by the immersion, absolute boundary conditions, and an orthogonal representation $\rho$ of the fundamental group, is ($p>0$):
\begin{align*}
\log T_{\rm abs}((D^{2p-1}_l,g_E);\rho)&=\frac{1}{2}{\rm rk}(\rho)\log{\rm Vol}_{g_E}(D_l^{2p-1})+\frac{1}{2}{\rm rk}(\rho)\log 2+\frac{1}{4}{\rm rk}(\rho)\sum_{n=1}^{p-1}\frac{1}{n},\\
\log T_{\rm abs}((D^{2p}_l,g_E);\rho)&=\frac{1}{2}{\rm rk}(\rho)\log{\rm Vol}_{g_E}(D^{2p}_l)+
\frac{p}{2}{\rm rk}(\rho)\sum_{n=1}^p\frac{1}{2n-1}.
\end{align*}

\end{theo}

\begin{proof} The results follow from Theorem \ref{t1}, using Lemmas \ref{ddff} and \ref{df} of Section \ref{s3}.
\end{proof}

Beside our main results concern the case of the discs, that are smooth manifolds, our technique easily extends, at least formally, to cover the case of the completed metric cone of angle $\alpha$ over a sphere, $C_\alpha S^n_{l\sin\alpha}$ (see the beginning of Section \ref{s2} for the definition). And this generalization contains the case of the discs. For this reason, we develop our analysis in the more general case of the cone, whenever this is possible. The main problem, to deal with the cone, is the extension of the Hodge theory to the space of $L^2$-forms near the singularity at the tip of the cone. This theory has been developed in the work of J. Cheeger \cite{Che2}, and we will assume his results in the definition of the Laplacian on forms, necessary in order to define the analytic torsion appearing in the following theorems. More details on this aspect, are at the beginning of Section \ref{s3} and of Section \ref{s4}.

\begin{theo}\label{t3} The analytic torsion of the cone $C_\alpha S^1_{l\sin\alpha}$, of angle $\alpha$ and length $l>0$, over the circle, with the standard metric $g_E$ induced by the immersion, absolute/relative boundary conditions, and an orthogonal representation $\rho_0$ of the fundamental group of rank 1, is:
\begin{align*}
\log T_{\rm abs}((C_\alpha S^1_{l\sin\alpha},g_E);\rho_0)=-\log T_{\rm rel}=\frac{1}{2}\log(\pi l^2\sin\alpha)+\frac{1}{2}\sin\alpha.
\end{align*}

In particular, for the disc $D^2_l$ we have:
\begin{align*}
\log T_{\rm abs}((D^2_l,g_E);\rho_0)=-\log T_{\rm rel}((D^2_l,g_E);\rho_0)=\frac{1}{2}\log\pi l^2+\frac{1}{2}.
\end{align*}
\end{theo}
\begin{proof} The proof is in Section  \ref{s4.2}.
\end{proof}

\begin{theo}\label{t5} The analytic torsion of the cone $C_\alpha S^2_{l\sin\alpha}$, of angle $\alpha$ and length $l>0$, over the sphere, with the standard metric $g_E$ induced by the immersion, absolute/relative boundary conditions, and an orthogonal  representation $\rho_0$ of the fundamental group of rank 1, is:
\begin{align*}
\log T_{\rm abs}((C_\alpha S^2_{l\sin\alpha},g_E);\rho_0)=\log T_{\rm rel}=\frac{1}{2}\log\frac{4}{3}l^3-\frac{1}{2}F(0,\csc\alpha)+\frac{1}{4}\sin^2\alpha,
\end{align*}
where the function $F(0,x)$ is given in Appendix \ref{appendixB}.
In particular, for the disc $D^3_l$ we have:
\begin{align*}
\log T_{\rm abs}((D^3_l,g_E);\rho_0)=\log T_{\rm rel}((D^3_l,g_E);\rho_0)=\frac{1}{2}\log\frac{4\pi l^3}{3}+\frac{1}{2}\log 2+\frac{1}{4}.
\end{align*}

\end{theo}

\begin{proof} The proof is in Section  \ref{s4.3}.
\end{proof}

Some comments on these results are in order.
\begin{enumerate}

\item The geometric-topological approach (of Section \ref{s2}) proves to be much more effective and natural for evaluating the analytic torsion with respect to a direct calculation starting from the definition of the analytic torsion. The geometric-topological approach also gives a clear interpretation of the different terms appearing in the final result, as described in the next remarks (a similar result was obtained for the spheres in \cite{dMS}).

\item In the formula for the analytic torsion of the disc $D^2_l$ (obtained particularizing the second formula given in Theorem \ref{t2}), two terms appear. The first is the topological term, corresponding to the volume of the disc, and in fact comes from the topological part of the analytic torsion, namely the RS torsion, that can be computed using Theorem \ref{t1}. The second term comes from the boundary contribution. In this case, the unique boundary contribution is of the type described by Dai and Fang \cite{DF} and Br\"{u}ning and Ma \cite{BM}, since the manifold is not a product near the boundary, and there is no contribution of type described by L\"{u}ck \cite{Luc}, since the Euler characteristic of the boundary is trivial. 
The result given in Theorem \ref{t2} is obtained by using the formula of Br\"{u}ning and Ma, and is consistent with the result obtained by direct calculation of the analytic torsion, given in Theorem \ref{t5}. However, in this case the formula of Dai and Fang gives the same result, see the analysis at the end of Section \ref{s3}.

\item In the formula for the analytic torsion of the disc $D^3_l$ (obtained particularizing the first formula given in Theorem \ref{t2}), three terms appear. The first is the topological term, corresponding to the volume of the disc, as in the case of $D^2_l$. The second term comes from the boundary contribution, and is precisely the term depending on the Euler characteristic as predicted by the formula of Luck \cite{Luc}. The last term also must come from the boundary, and therefore comes from the fact that the manifold is not a product near the boundary. The result given in Theorem \ref{t2} is obtained by using the formula of Br\"{u}ning and Ma, and is consistent with the result obtained by direct calculation of the analytic torsion, given in Theorem \ref{t5}. In this case, it is easy to check that all the contributions arising from the formula in Theorem 1 of \cite{DF} vanish in this case (see formula (\ref{ddff2}) at the end of Section \ref{s3}). Thus, this result furnishes a counter example to the theorem of Dai and Fang, at least for the even case (see also Remark \ref{rema1} at the end of Section \ref{s3}).

\item An easy calculation (see for example \cite{Luc} (1.19), where a multiplicative  factor 2 appears, arising from the different definition of the analytic torsion) shows that the analytic torsion of the one dimensional disc with absolute boundary conditions is
\[
\log T_{\rm abs}((D^1_l,g_E),\rho_0)=\frac{1}{2}\log l+\frac{1}{2}\log 2.
\]

This is consistent with Theorem \ref{t2}, since in this case the metric is a product near the boundary, and therefore the generalization of the Cheeger-M\"{u}ller theorem proved by Luck in \cite{Luc} reads $T(W)=\tau(W)+\frac{1}{4}\chi(\b W)\log 2$.

\item Using Theorem \ref{t3} and Proposition \ref{p2}, the ratio of the analytic and Reidemeister torsion on the cone over the disc reads
\[
\log T(C_\alpha S^1_{l\sin\alpha})-\log\tau(C_\alpha S^1_{l\sin\alpha})=\frac{1}{2}\sin\alpha.
\]

On the other side, using the formula given in Remark  \ref{rema}, the boundary term in this case is
precisely $\frac{1}{2}\sin\alpha$. Therefore, in the formula for the analytic torsion there are no terms coming from the singularity, or in other words in this case the analytic torsion does not detect the presence of the singularity. A torsion type invariant that detect the presence of the singularity was introduced in \cite{Spr6}, by considering the Laplacian with Dirichlet boundary conditions on forms.
We can not perform a similar analysis in the case of the sphere, since we do not have a satisfactory description of the boundary term in that case.
\end{enumerate}

\section{Preliminary and notation}
\label{s1}

In this section we recall the basic fact about torsion and we introduce some notation. We prefer to introduce the R torsion starting from the more abstract Whitehead (W) torsion. From one side, this makes the presentation far more general, and far more useful for further applications we are working on, from the other this allows to describe in a concise way the main properties of invariance of the  R torsion, necessary for all our applications. For the same reasons, we consider the general case of a non simply connected space. This section is essentially based on \cite{Mil}, \cite{Coh} and \cite{RS}.

Let $R$ be an associative ring with unit.
Let $K_1(R)$ be the Whitehead group of $R$, i.e. the abelianization of the general linear group $Gl(R)$. Let $\bar K_1(R)=K_1(R)/\{[1],[-1]\}$ be the reduced Whitehead group. Let
$$
\xymatrix{
C:& C_m\ar[r]^{\b_m}&C_{m-1}\ar[r]^{\b_{m-1}}&\dots\ar[r]^{\b_2}&C_1\ar[r]^{\b_1}&C_0,
}
$$
be a chain complex of (finite dimensional) free $R$-modules\footnote{We will assume that if $M$ is any finitely generated module over $R$ then any two bases of $M$ have the same cardinality.}. Denote by $Z_q=\ker \b_q$, by $B_q=\Im \b_{q+1}$, and  by
$H_q(C)=Z_q/B_q$ as usual. Assume that both $B_q$ and $H_q$ are free $R$-modules  for each $q$. For two bases $x=\{x_1,\dots, x_k\}$ and
$y=\{y_1,\dots, y_k\}$ of a free $R$-module $M$, denote by $(y/x)$ the matrix defined by the change of bases, and by $[y/x]$ the class of this matrix
in $\bar K_1 (R)$. For each $q$, fix a base $c_q$ for $C_q$, and a base $h_q$ for $H_q(C)$. Let $b_q$ be a (independent) set of elements of $C_q$ such that
$\b_q(b_q)$ is a base for $B_{q-1}$. Then the set of elements $\{\b_{q+1}(b_{q+1}),h_q, b_q\}$ is a base for $C_q$. In this situation, the
Whitehead torsion of the complex $C$ with respect to the graded base $h=\{h_q\}$ is the class
\[
\tau_{\rm W}(C;h)=\sum_{q=0}^m (-1)^q [\b_{q+1}(b_{q+1}),h_q,b_q/c_q],
\]
in $\bar K_1 (R)$.

Let $\Z G$ be the group ring of a given group $G$. Then, $G$ embeds homomorphically in the group of units $(\Z G)^\times =Gl(1,\Z G)\subset Gl(k,\Z G)$. This immersion passes to the quotient defining an homomorphism of $G$ into a normal subgroup of $\bar K_1(\Z G)$. The quotient space is the Whitehead group $Wh(G)$ of $G$, and we denote by $w:\bar K_1(\Z G)\to Wh(G)$ the natural projection.

Let $(K,L)$ be a pair of connected finite cell complexes of dimension $m$, and $(\tilde K,\tilde L)$ its universal covering complex pair, and identify the fundamental group of $K$
with the group of the covering transformations of $\tilde K$.  Let $C((\tilde K,\tilde L);R)$ be the chain complex of $(\tilde K,\tilde L)$ with coefficients in $R$. The action of the
group of covering transformations makes each chain group $C_q((\tilde K,\tilde L);R)$ into a module over the group ring $R\pi_1(K)$, and each of these modules
is $R \pi_1(K)$-free and finitely generated by the natural choice of the $q$-cells of $K-L$. Since $K$ is finite it follows that $C((\tilde K,\tilde L);R)$ is
free and finitely generated over $R \pi_1(K)$. We have got a complex of free finitely generated modules over $R \pi_1(K)$ that, following standard notation, we denote by $C((\tilde K,\tilde L);R\pi_1(K))$ of free finitely generated modules over $R \pi_1(K)$.
Any fixed choice of a graded basis $c$ for $C((\tilde K,\tilde L);R \pi_1(K))$ and of  a graded basis of $H(C((\tilde K,\tilde L); R \pi_1(K)))$, allows to define the torsion $\tau_{\rm W}(C((\tilde K,\tilde L);R\pi_1(K));h)$ in $\bar K_1(R\pi_1(K))$. However, there is still some ambiguity due to the arbitrariety of the choice of the representative cells in the covering space $\tilde K$, projecting over the cells representing the fixed bases of $C_q$. Taking integer coefficients, the different choices of representative cells in the covering give different torsions in $\bar K_1(\Z \pi_1(K))$ that however project to the same class in $Wh(\pi_1(K))$.
Assuming $H_q(C((\tilde K,\tilde L); \Z \pi_1(K)))$ are free finitely generated modules over $\Z \pi_1(K)$, the Whitehead torsion of $K$ with respect to the graded base $h$ is the class
\[
\tau_{\rm W}((K,L);h)=w(\tau_{\rm W}(C((\tilde K,\tilde L);\Z\pi_1(K));h)),
\]
in $Wh(\pi_1(K))$. If $(K,L)$ is the cellular (or simplicial)
decomposition of a space $(X,A)$, the Whitehead torsion of $(X,A)$ is defined accordingly, and denoted by $\tau_{\rm W}((X,A);h)$. It was proved in \cite{Mil} that $\tau_{\rm W}((X,A);h)$ does not depend on the decomposition $K$.

If $\varphi:R\to R'$ is a ring homomorphism, we can form a new free complex $R'\otimes_R C_q$,
using the homomorphism $\varphi$ to make $R'$ into a right $R$-module. Then, $\tau_{\rm W}(R'\otimes_R C;h')=\varphi_* \tau_{\rm W}(C(R);h)$. This is used to define the torsion with respect to a representation of the fundamental group. Let $\rho:\pi_1(X)\to G$ be a representation of the fundamental group in some group $G$. Then, $\rho$ extends to a unique ring homomorphism from $\Z\pi_1(X)$ to $\Z G$, and we form the complex
\[
C_q((X,A);(\Z\pi_1(K))_\rho)=\Z G\otimes_{\rho} C_q((\tilde X,\tilde A);\Z\pi_1(K)),
\]
of free finitely generated $\Z G$-modules, and the Whitehead torsion of $(X,A)$ with respect to the representation $\rho$ (and the graded basis $h$) as
\[
\tau_{\rm W}((X,A);h,\rho)=w(\tau_{\rm W}(C((X,A);(\Z\pi_1(X)_\rho));h)),
\]
in $Wh(\pi_1(X))$.

This construction is particularly usefull when $G$ is a subgroup of the group of units of a field. For example, let $G=O(k,\R)$. Then, the determinant function $\det:K_1(\Z O(k,\R))\to \R^\times$, induces an isomorphism of $Wh(O(k,\R))=\R^+$, the positive real numbers. In this situation, we define the R torsion by $\tau_{\rm R}((X,A);h,\rho)=\det \tau_{\rm W}((X,A);h,\rho)$. The multiplicative notation is more convenient in this case, thus we have
\beq\label{e1}
\tau_{\rm R}((X,A);h,\rho)=\prod_{q=0}^m \left| \det\rho(\b_{q+1}(b_{q+1}),h_q,b_q/c_q)\right|^{(-1)^q},
\eeq
in $\R^+$.

Next, let $W$ be a connected orientable Riemannian manifold of dimension $m$ with possible boundary $\partial W$, and Riemannian metric $g$. Then, all the previous assumptions are satisfied, and we can define an absolute R torsion  $\tau_{\rm R}((W,\emptyset);h,\rho)$, and a relative R torsion  $\tau_{\rm R}((W,\b W);h,\rho)$, for each representation $\rho$ of the fundamental group, and for each fixed graded base $h$ for the homology of $(W,\emptyset)$, and for the relative homology of $(W,\b W)$, respectively. In this context, Ray and Singer suggest a natural geometric invariant object, by fixing an appropriate base $h$ using the geometric
structure, as follows.

Let $E_\rho\to W$ be the real vector bundle associated to the representation $\rho:\pi_1(W)\to O(k,\R)$. Let
$\Omega(W,E_\rho)$ be the graded linear space of smooth forms on $W$ with values in $E_\rho$. A base for $\Omega^q(W, E_\rho)$ is of the form
$\{C^\infty(W)\otimes dx_I\otimes_\rho e_i\}$. The exterior differential on $W$ defines the exterior differential on $\Omega^q(W, E_\rho)$, $d:\Omega^q(W, E_\rho)\to
\Omega^{q+1}(W, E_\rho)$. The metric $g$ defines an Hodge operator on $W$ and hence on $\Omega^q(W, E_\rho)$, $*:\Omega^q(W, E_\rho)\to
\Omega^{n-q}(W, E_\rho)$, and,  using the inner product in $E_\rho$,
an inner product on $\Omega^q(W, E_\rho)$,
as follows. Let
\begin{align*}
&\omega=\sum_{I,j} \omega_{I,j}(x) d x^I\otimes e_j,&
\eta=\sum_{I',j'} \eta_{I',j'}(x) d x^{I'}\otimes e_{j'},\\
\end{align*}
then:
\beq\label{inner}
(\omega,\eta)=\sum_{I,I',j,j'}\int_W \omega_{I,j}(x)\omega_{I',j'}(x)dx^I\wedge *d x^{I'}(e_j,e_{j'}),
\eeq
where $(v,w)=v^T w$ is the inner product in $\R^k$, and $\wedge$ is the exterior product in the de Rham bundle of forms on $W$, $\Lambda W$.

In order to proceed we recall some results on manifolds with boundary. If $W$ has a boundary $\b W$, then there is a natural splitting of $\Lambda W$ as direct sum of vector bundles $\Lambda \b W\oplus N^* W$, where $N^*W$ is the dual to the normal bundle to the boundary. Locally, this reads as follows. Let $\b_r$ denotes the outward pointing unit normal vector to the boundary, and $dr$ the correspondent one form. Near the boundary we have the collar decomposition
$C(\b W)=[0,-\epsilon)\times \b W$, and if $y$ is a system of local coordinates on the boundary, then $x=(r,y)$ is a local system of coordinates in $C(\b W)$. The smooth forms on $W$ near the boundary decompose as
\[
\omega=\omega_{\rm tan}+\omega_{\rm norm},
\]
where $\omega_{\rm norm}$ is the orthogonal projection on the subspace generated by $dr$, and $\omega_{\rm tan}$ is in $\Lambda \b W$. We  write
\[
\omega=\omega_1+ \omega_{2}\wedge dr,
\]
where $\omega_j\in C^\infty(\b W)\otimes \Lambda \b W$, and
\beq\label{dec}
*\omega_2=dr \wedge *\omega.
\eeq

Define absolute boundary conditions by
\[
B_{\rm abs}(\omega)=\omega_{\rm norm}|_{\b W}=\omega_1|_{\b W}=0,
\]
and relative boundary conditions by
\[
B_{\rm rel}(\omega)=\omega_{\rm tan}|_{\b W}=\omega_2|_{\b W}=0.
\]

Let $\B(\omega)=B(\omega)\oplus B((d+d^\dagger)(\omega))$. Then the operator $\Delta=(d+d^\dagger)^2$ with boundary conditions $\B(\omega)=0$  is self adjoint. Note that $\B$ correspond to
\beq\label{abs}
\B_{\rm abs}(\omega)=0\hspace{20pt}{\rm if~ and~ only~ if}\hspace{20pt}\left\{\begin{array}{l}\omega_{\rm norm}|_{\b W}=0,\\
(d\omega)_{\rm norm}|_{\b W}=0,\\
       \end{array}
\right.
\eeq
\beq\label{rel}
\B_{\rm rel}(\omega)=0\hspace{20pt}{\rm if~ and~ only~ if}\hspace{20pt}\left\{\begin{array}{l}\omega_{\rm tan}|_{\b W}=0,\\
(d^\dagger\omega)_{\rm tan}|_{\b W}=0,\\
       \end{array}
\right.
\eeq


Let
\begin{align*}
\H^q(W,E_\rho)&=\{\omega\in\Omega^q(W,E_\rho)~|~\Delta^{(q)}\omega=0\},\\
\H_{\rm abs}^q(W,E_\rho)&=\{\omega\in\Omega^q(W,E_\rho)~|~\Delta^{(q)}\omega=0, B_{\rm abs}(\omega)=0\},\\
\H_{\rm rel}^q(W,E_\rho)&=\{\omega\in\Omega^q(W,E_\rho)~|~\Delta^{(q)}\omega=0, B_{\rm rel}(\omega)=0\},
\end{align*}
be the spaces of harmonic forms with boundary conditions. Then we have the following de Rham maps $\A^q$ (that induce isomorphisms in cohomology),
\begin{align*}
\A^q:&\H^q(W,E_\rho)\to C^q(W;E_\rho),\\
\A_{\rm rel}^q:&\H_{\rm rel}^q(W ,E_\rho)\to C^q((W,\b W);E_\rho),\\
\end{align*}
with
\[
\A^q(\omega)(c\otimes_\rho v)=\A^q_{\rm rel}(\omega)(c\otimes_\rho v)=\int_{ c} (\omega,v),
\]
where $c\otimes_\rho v$ belongs to $C^q(W;E_\rho)$ in the first case, and belongs to $C^q((W,\b W);E_\rho)$  in the second case, and $c$ is identified with the $q$-cell (simplicial or cellular) that $c$ represents.

Next, let $K$ be a cellular or simplicial decomposition of $W$ and $L$ of $\b W$. Let $sd K$ be the first barycentric subdivision of $K$. Let $\hat K$ be the dual block complex of $K$. For each $q$-cell (simplex) $c$ in $K$, let $\hat c$ denotes the $(m-q)$-dual block of $c$. If we take $q$-cells $c$ in $K-L$, then the collection of the cells $\hat c$ is a base for the relative chain group $C_q(K,L;\Z)$. Therefore the bijection $c\to \hat c$, induces a bijection between the bases of $C_q(K,L;\Z)$ and $C_{m-q}(K;\Z)$, and therefore an isomorphism (the image is in $\hat K-\hat L$ since $\hat c$ is disjoint from $L$ if $c\in K-L$) commuting with boundary operators (up to the sign factor $(-1)^{m-q+1}$)
\begin{align*}
\varphi_q:&C_q(K,L;\Z)\to C^{m-q}(\hat K-\hat L;\Z),\\
\varphi_q:&c\mapsto \hat c.
\end{align*}

Dualizing, let $c^*$ denotes the cochain dual to the chain $c$. Then, we have the Lefschetz-Poincar\'e isomorphisms (that induces isomorphisms from homology onto cohomology)
\begin{align*}
\P_q:&C_q(K,L;\Z)\to C^{m-q}(\hat K-\hat L;\Z),\\
\P_q:&c\mapsto (\hat c)^*.
\end{align*}

Following Ray and Singer, we introduce the de Rham maps $\A_q$:
\begin{align*}
\A^{\rm rel}_q:&\H^q(W,E_\rho)\to C_q((W,\b W);E_\rho),\\
\A^{\rm rel}_q:&\omega\mapsto (-1)^{(n-1)q}\P_q^{-1}\A^{n-q}*(\omega),\\
\A^{\rm abs}_q:&\H_{\rm abs}^q(W,E_\rho)\to C_q(W;E_\rho),\\
\A^{\rm abs}_q:&\omega\mapsto (-1)^{(n-1)q}\P_q^{-1}\A_{\rm rel}^{n-q}*(\omega),
\end{align*}
both defined by
\beq\label{aa}
\A^{\rm abs}_q(\omega)=\A^{\rm rel}_q(\omega)=(-1)^{(n-1)q}\sum_{j,i} \left(\int_{\hat c_{q,j}}(*\omega,e_i)\right)
c_{q,j}\otimes_\rho e_i,
\eeq
where the sum runs over all $q$-simplices $c_{q,j}$ of $W$ in the first case, but runs over all $q$-simplices $c_{q,j}$ of $W-\b W$ in the second case.

In this situation, let $a$ be a graded orthonormal base for the space of the harmonic forms in $\Lambda W \otimes_\rho \R^k$. Then following Ray and Singer \cite{RS} Definition 3.6,  we call the positive real number
\[
\tau_{\rm RS}((W,g);\rho)=\tau_{\rm R}(W;\A^{\rm abs}(a),\rho),
\]
the Ray-Singer (RS) torsion of $(W,g)$ with respect to the representation $\rho$, and absolute boundary condition, and
\[
\tau_{\rm RS}((W,\b W,g);\rho)=\tau_{\rm R}((W,\b W);\A^{\rm rel}(a),\rho),
\]
the Ray-Singer (RS) torsion of $(W,\partial W,g)$ with respect to the representation $\rho$, and relative boundary condition.  It is possible to prove that  both $\tau_{\rm RS}((W,g);\rho)$ and $\tau_{{\rm RS}}((W,\b W,g);\rho)$ do not depend
on the choice of the orthonormal base $a$.

Going back along the previous construction, we find out that the formula in equation (\ref{e1}) gives in the present situation the following formula
for the RS torsion (where $N$ can be either $\b W$ or the empty set, and $\A_q$ is the relative or the absolute one, respectively)
\beq\label{e2}
\tau_{\rm RS}((W,N,g);\rho)=\prod_{q=0}^m \left| \det\rho(\b_{q+1}(b_{q+1}),\A_q(a_q),b_q/c_q)\right|^{(-1)^q},
\eeq
where $\det$ is the usual determinant of the matrix of the change of bases 
in the complex of real vector spaces $C((W,N);E_\rho)$:
\beq\label{cc}
\xymatrix{
C_m\otimes_\rho \R^k\ar[r]^{\hspace{16pt}\b_m\otimes_\rho 1}&\dots\ar[r]^{\hspace{-15pt}\b_2\otimes_\rho 1}&C_1\otimes_\rho \R^k\ar[r]^{\b_1\otimes_\rho 1}&C_0\otimes_\rho \R^k.
}
\eeq

We conclude this section with the definition of the analytic torsion. First assume $W$ has no boundary. With the inner product defined in equation (\ref{inner}), $\Omega(W,E_\rho)$ is an Hilbert space. Let $d^\dagger=(-1)^{mq+m+1}*d*$ be the formal adjoint of $d$, then the Laplacian $\Delta=(d^\dagger d+d d^\dagger)$ is a symmetric non negative definite operator in $\Omega(W,E_\rho)$, and has pure point spectrum $\Sp \Delta$. Let $\Delta^{(q)}$ be the restriction of $\Delta$ to $\Omega^q(W,E_\rho)$. Then we define the zeta function of $\Delta^{(q)}$ by the series
\[
\zeta(s,\Delta^{(q)})=\sum_{\lambda\in \Sp_+ \Delta^{(q)}} \lambda^{-s},
\]
for $\Re(s)>\frac{m}{2}$, and where $\Sp_+$ denotes the positive part of the spectrum. The above series converges uniformly for $\Re(s)>\frac{m}{2}$, and extends to a meromorphic function analytic at $s=0$. Following Ray and Singer \cite{RS} Definition 7.2, we define the analytic torsion of $(W,g)$ by
\beq\label{analytic}
\log T((W,g);\rho)=\frac{1}{2}\sum_{q=1}^m (-1)^q q \zeta'(0,\Delta^{(q)}).
\eeq

If $W$ has a boundary, we denote by $T_{\rm abs}((W,g);\rho)$ the number defined by equation (\ref{analytic}) with $\Delta$ satisfying absolute boundary conditions, and by $T_{\rm rel}(W,g);\rho)$ the number defined by the same equation with $\Delta$ satisfying relative boundary conditions.

\section{The RS torsion of the geometric cone over a sphere}
\label{s2}

In this section we compute the RS torsion of the $m$-dimensional disc, $D^m_l$. However, we will consider the slightly more general case of a cone. Namely we consider the cone of angle $\alpha$, $C_\alpha S^{n}$, constructed in $\R^{n+2}$ over the sphere $S^{n}$, $m=n+1$, as defined below. It turns out that $C_\alpha S^{n}$ is not a smooth Riemannian manifold, but is a space with a singularity of conical type as defined by Cheeger in \cite{Che1} (2.1). More precisely, $C_\alpha S^n$ coincides with the completed finite metric cone of Cheeger over the sphere of radius $\sin\alpha$. Note that we are adding a point at the tip of the cone, in order to have a simply connected space. The resulting space is compact, but obviously is not a smooth Riemannian manifold. The space obtained removing the tip, is an open complete smooth Riemannain manifold with the metric induced by the immersion, as in \cite{Che1}. It is clear that we can define the W torsion, and the R torsion on $C_\alpha S^n$. For what concerns the RS torsion, some care is necessary, since we do not know how the de Rham theory extends. More precisely, we do not know if we have the de Rham maps $\A^q$ for spaces with conical singularities, in general. However, we show that we can define these maps in the particular case of $C_\alpha S^{n}$, and therefore we define the RS torsion accordingly. In particular, the construction cover the smooth case of the disc.

Let $S^n_b$ be the standard sphere of radius $b>0$ in $\R^{n+1}$, $S^{n}_b=\{x\in\R^{n+1}~|~|x|=b\}$ (we simply write $S^n$ for $S^n_1$). Let $C_\alpha
S^n_{l\sin\alpha}$ be the cone of angle $\alpha$ over $S^n_{l\sin\alpha}$ in $\R^{n+2}$. Note that the disc corresponds to $D^{n+1}_l=C_\frac{\pi}{2} S^{n}_l$. We parameterize $C_{\alpha}S^n_{l\sin\alpha}$ by
\begin{equation*}\label{}C_{\alpha}S_{l\sin\alpha}^{n}=\left\{
\begin{array}{rcl}
x_1&=&r \sin{\alpha} \sin{\theta_n}\sin{\theta_{n-1}}\cdots\sin{\theta_3}\sin{\theta_2}\cos{\theta_1} \\[8pt]
x_2&=&r \sin{\alpha} \sin{\theta_n}\sin{\theta_{n-1}}\cdots\sin{\theta_3}\sin{\theta_2}\sin{\theta_1} \\[8pt]
x_3&=&r \sin{\alpha} \sin{\theta_n}\sin{\theta_{n-1}}\cdots\sin{\theta_3}\cos{\theta_2} \\[8pt]
&\vdots& \\
x_{n+1}&=&r \sin{\alpha} \cos{\theta_n} \\[8pt]
x_{n+2}&=&r \cos{\alpha}
\end{array}
\right.\end{equation*}
with $r \in [0,l]$, $\theta_1 \in [0,2\pi]$, $\theta_2,\ldots,\theta_n \in [0,\pi]$, $\alpha$ is a fixed positive real number, and
$0<a=\frac{1}{\nu}= \sin{\alpha}\leq 1$. The induced metric is ($r>0$)
\begin{align*}
g_E &=dr \otimes dr + r^2 a^2 g_{S^{n}_{1}}\\
&= dr\otimes dr + r^2 a^2 \left(\sum^{n-1}_{i=1} \left(\prod^{n}_{j=i+1} \sin^2{\theta_j}\right) d\theta_i \otimes
d\theta_i + d\theta_n \otimes d\theta_n\right),
\end{align*}
and $\sqrt{|\det g_E|}=(r\sin\alpha)^{n}(\sin\theta_n)^{n-1}(\sin\theta_{n-1})^{n-2}\cdots(\sin\theta_3)^{2}(\sin\theta_2)$.
Let $K$ be the cellular decomposition of $C_{\alpha}S^{n}_{l\sin\alpha}$ , with one top cell, one $n$-cell and one $0$-cell,
$K=c^1_{n+1} \cup c^1_{n} \cup c^1_0$. Let the subcomplex $L$ of $K$ be the cellular decomposition of $S^{n}_{l\sin\alpha}$, $L
= c^1_n \cup c^1_0$. 

We consider first the case of relative boundary conditions. Then the complex of real vector spaces of equation (\ref{cc}) reads
\[
C_{\rm rel}: \xymatrix{0\ar[r] & \mathbb{R}[c^1_{n+1}] \ar[r] & 0\ar[r]&\cdots \ar[r]&0\ar[r] & 0 \ar[r] & 0},
\]
with preferred base $c_{n+1} = \{c^{1}_{n+1}\}$. To fix the base for the homology, we need a graded orthonormal base
$a$ for the harmonic forms. Since a base for $\Omega^{n+1}(C_{\alpha}S^{n}_{l\sin\alpha})$ is $\{\sqrt{|\det g_E|}dr \wedge
d\theta_1\wedge\cdots\wedge d\theta_n\}$, we get $a_{n+1} = \left\{\frac{\sqrt{|\det g_E|}dr
\wedge d\theta_1\wedge\cdots\wedge d\theta_n}{\sqrt{\Vol_{g_E}(C_{\alpha}S^n_{l\sin\alpha})}} \right\}$. Applying the formula in equation (\ref{aa}) for the de Rham map,  we obtain $h_{n+1}=\{h_{n+1}^1\}$, with
\begin{align*}
h^1_{n+1} &=\A^{\rm rel}_{n+1}(a^1_{n+1})= \frac{1}{\sqrt{\Vol_{g_E}(C_{\alpha}S^{n}_{l\sin\alpha})}}\int_{\rm pt}\ast \sqrt{|\det g_E|}dr\wedge
d\theta_1\wedge\cdots\wedge d\theta_n c^1_{n+1}\\
&= \frac{1}{\sqrt{\Vol_{g_E}(C_{\alpha}S^{n}_{l\sin\alpha})}}  c_{n+1}^1.
\end{align*}

As $b_q = \emptyset$, for all $q$, we have that
\[
|\det (h_{n+1}/c_{n+1})| = \frac{1}{\sqrt{\Vol_{g_E}(C_{\alpha}S^{n}_{l\sin\alpha})}}, \qquad |\det (b_{q}/c_{q})| =
1,~ 0\leq q\leq n.
\]

Applying the definition in equation (\ref{e2}), this proves the following result.

\begin{prop}\label{px}
\begin{align*}
\tau_{\rm RS}((C_{\alpha}S^{n}_{l\sin\alpha},S^n_{l\sin\alpha},g_E);\rho)  = & \left(\sqrt{\Vol_{g_E}(C_{\alpha}S^{n}_{l\sin\alpha})}\right)^{(-1)^{n}{\rm rk}(\rho)}. \\
\end{align*}
\end{prop}

Next, we consider the case of absolute boundary conditions. By equation (\ref{cc}), the relevant complex is
\[
C_{\rm abs}: \xymatrix{0\ar[r] & \mathbb{R}[c^1_{n+1}] \ar[r] & \mathbb{R}[c^{1}_{n}]\ar[r]&0\ar[r]&\cdots \ar[r]&0\ar[r]&
\mathbb{R}[c^1_0]\ar[r] & 0},
\]
with preferred bases $c_{n+1} = \{c^{1}_{n+1}\}$, $c_{n} = \{c^{1}_{n}\}$ and $c_{0} = \{c^{1}_{0}\}$. Hence,
$H_{p}(K) = 0$, for $p>1$, and $H_{0}(K) = \R[c^1_0]$. Since a base for $\Omega^{0}(C_{\alpha}S^{n}_{l\sin\alpha})$ is the constant form $\{1\}$, we have $a_{0}
= \left\{\frac{1}{\sqrt{\Vol_{g_E}(C_{\alpha}S^n_{l\sin\alpha})}}\right\}$. Applying the formula in equation (\ref{aa}) for the de Rham map,  we obtain $h_{0}=\{h_{0}^1\}$, with
\begin{align*}
h^1_{0} &=\A_{0}^{\rm abs}(a^1_{0})= \frac{1}{\sqrt{\Vol_{g_E}(C_{\alpha}S^{n}_{l\sin\alpha})}}\int_{C_{\alpha}S^{n}_{{l\sin\alpha}}}\ast 1 c^1_{0}\\
&= \sqrt{\Vol_{g_E}(C_{\alpha}S^{n}_{l\sin\alpha})} c_{0}^1.
\end{align*}

As $b_q =\emptyset$ for $q=0,\ldots,n$, $b^1_{n+1} = c^1_{n+1}$ and $\b (b_{n+1}^1)=c_n$, we have
that
\begin{align*}
|\det (h_{0}/c_{0})| &= \sqrt{\Vol_{g_E}(C_{\alpha}S^{n}_{l\sin\alpha})},&\\
 |\det (\b (b_{n+1}^1)/c_n)| &= 1, &|\det (b_{n+1}/c_{n+1})| =1.
\end{align*}

Applying the definition in equation (\ref{e2}), this proves the following result.

\begin{prop}\label{p2}
\begin{align*}
\tau_{\rm RS}((C_{\alpha}S^{n}_{l\sin\alpha},g_E);\rho)  = & \left(\sqrt{\Vol_{g_E}(C_{\alpha}S^{n}_{l\sin\alpha})}\right)^{{\rm rk}(\rho)}. \\
\end{align*}
\end{prop}

\section{The anomaly boundary term and the analytic torsion of a disc}
\label{s3}

The aim of this section is to give the proof of Theorem \ref{t2}. For we need a formula for  the ratio between the analytic torsion and the Reidemeister torsion, that we call anomaly boundary contribution.
Observe that, by result of Cheeger \cite{Che1}, this ratio depends only on some geometric terms coming from the geometry of the manifold near the boundary. Since the singularity at the tip of the cone does not affect the geometry near the boundary, we are allowed to perform our calculation for the general case of the cone $C_{\alpha} S^n_{l\sin\alpha}$. However, this would imply some technical difficulties that are beside the aim of the present work, and will be tackled in a forecoming paper.
As observed in the introduction, two different formulas for this anomaly are available at the moment. One is given in Theorem 1 of \cite{DF}, and the second one comes from Theorem 1 of \cite{BM}. We first proceed to evaluate the anomaly boundary contribution using the formula of \cite{BM}. Then, at the end of the section, we will describe the contribution appearing using the formula of \cite{DF}.

We proceed in two steps. 
First we give in Lemma \ref{ddff} formulas for the anomaly in terms of some geometric invariants. This follows directly from Theorem  1 of \cite{BM}, and gives, in the odd dimensional case, the anomaly in terms of the Euler characteristic of the boundary. The even case is harder, and needs the introduction of some machinery and notation from \cite{BM} and \cite{BZ}. This is done in the course of the proof, and, as a result, the anomaly in the even case is written as some integral. The second step is accomplished in Lemma \ref{df}, where we give all the geometric invariants  necessary to compute the integral appearing in the formula obtained in Lemma \ref{ddff}, and we conclude the calculation for the even case.

Before to start, we need some notation, that will be used without further comments in this section. The parameterization of the cone and the induced metric $g_E$ are given in Section \ref{s2}. Define the metrics
\begin{align*}
g_1=g_E &= dr \otimes dr + r^2 g_{S^{n}},\\
g_0& = dr\otimes dr + l^{2} g_{S^{n}}.
\end{align*}

Let $\omega_{j}$, $j=0,1$, be the connection one forms associated to the metric $g_{j}$, and $\Omega_{j}=d\omega_{j}+\omega_{j}\wedge \omega_{j}$ the curvature two forms. 
Let $e(W,g)$ denotes the Euler class of $(W,g)$.


\begin{lem}\label{ddff} The ratio of the analytic torsion and the Reidemeister torsion of the disc $D_l^m$, of $m=2p-1$ odd dimension, and $m=2p$ even dimension ($p>0$),  with absolute boundary conditions are, respectively:
\begin{align*}
\log \frac{T_{\rm abs}((D^{2p-1}_l,g_E);\rho)}{\tau_{\rm
RS}((D^{2p-1}_l,g_E);\rho)}
&=\frac{1}{4}{\rm rk}(\rho)\chi(S^{2p-2}_l,g_E)\left(\log 2+\frac{1}{2} \sum_{n=1}^{p-1}\frac{1}{n}\right),\\
\log \frac{T_{\rm abs}((D^{2p}_{l},g_E);\rho)}{\tau_{\rm
RS}((D^{2p}_{l},g_E);\rho)} &=\frac{1}{2}{\rm
rk}(\rho)\frac{2^{p-1}}{\sqrt{\pi}(2p-1)!!} \sum_{j=1}^{p}
\frac{1}{2j-1} \int_{S^{2p-1}_l} \int^B \mathcal{S}_1^{2p-1}.\\
\end{align*}
\end{lem}
\begin{proof} The proof is based on Theorem 0.1 of \cite{BM}. Note that we are in the particular case of the flat trivial bundle $F$, since the unique representations are the trivial ones.
Therefore, we have from equation (0.6) and Section 4.1 of \cite{BM},
\beq\label{bat}
\begin{aligned}
\log \frac{T_{\rm abs}(D^{m}_{l},g_1);\rho)}{T_{\rm abs}((D^{m}_{l},g_0);\rho)}=
\frac{1}{2}{\rm rk}(\rho)\int_{S^{m-1}_l} \left(B(\nabla_1^{T D^{m}_{l}})-B(\nabla_0^{T D^{m}_{l}})\right),
\end{aligned}
\eeq
where the forms $B(\nabla_j^{TX})$ are defined in equation
(1.17) of \cite{BM} (see equation \ref{ebm1} below, and observe that we take the opposite sign with respect to the definition in \cite{BM}, since we are considering left actions instead of right actions). 
Since $g_0$ is a product near the boundary, by the result \cite{Luc}
\[
\log \frac{T_{\rm abs}(D^{m}_{l},g_0);\rho)}{\tau_{\rm RS}((D^{m}_{l},g_E);\rho)}=\frac{1}{4}{\rm rk}(\rho)\chi(S^{m-1},g_E)\log 2,
\]
it just remains to evaluate the anomaly boundary term, on the right side of equation (\ref{bat}). For, we first recall some notation from \cite{BZ}  Chapter III and \cite{BM} Section 1.1. For two $\Z/2$-graded algebras $\A$ and $\B$,
let $\A\hat\otimes\B=\A\wedge\hat\B$ denotes the $\Z/2$-graded
tensor product. For two real finite dimensional vector spaces $V$
and $E$, of dimension $m$ and $n$, with $E$  Euclidean and oriented,
the Berezin integral  is the linear map
\begin{align*}
\int^B&: \Lambda V^* \hat\otimes   \Lambda E^* \to \Lambda V^*, \\
\int^B:&\alpha \hat\otimes \beta\mapsto \frac{(-1)^{\frac{n(n+1)}{2}}}{\pi^\frac{n}{2}}\beta(e_1,\dots, e_n)\alpha,
\end{align*}
where $\{e_j\}_{j=1}^n$ is an orthonormal base of $E$. Let $A$ be an antisymmetric endomorphism of $E$. Consider the map
\[
{\hat{}} :A\mapsto \hat A=\frac{1}{2} \sum_{j,l=1}^n (e_j,A e_l)
\hat e^j\wedge \hat e^l.
\]

Note that 
\beq\label{pfpf} 
\int^B \e^{-\frac{\hat A}{2}}=Pf\left(\frac{A}{2\pi}\right), 
\eeq 
and this vanishes if ${\rm dim}E=n$ is odd.

Let $\omega_j$ be the curvature one form over $D^m_l$ associated to
the metric $g_j$. Let $\Theta$ be the curvature two form of the
boundary $S^{m-1}$ (with radius 1) and the standard Euclidean
metric. Let $\tensor{(\omega_j)}{^{a}_{b}}$ denotes the entries with
line $a$ and column $b$ of the matrix of one forms $\omega_j$.
Then, introduce the following quantities (see \cite{BM} equations
(1.8) and (1.15))
\beq\label{pippo}\begin{aligned}
\mathcal{S}_j&=\frac{1}{2}\sum_{k=1}^{m-1}\tensor{(\omega_j-\omega_0)}{^{r}_{\theta_k}}\hat e^{\theta_k},\\
\mathcal{R}&=\hat \Theta=\frac{1}{2}\sum_{k,l=1}^{m-1}\tensor{\Theta}{^{\theta_k}_{\theta_l}} \hat e^{\theta_k}\wedge \hat e^{\theta_l}.\\
\end{aligned}
\eeq


Then, we define
\beq\label{ebm1} 
B(\nabla_j^{T
D^{m}_{l}})=\frac{1}{2}\int_0^1\int^B
\e^{-\frac{1}{2}\mathcal{R}-u^2 \mathcal{S}_j^2}\sum_{k=1}^\infty
\frac{1}{\Gamma\left(\frac{k}{2}+1\right)}u^{k-1} \mathcal{S}_j^k
du. 
\eeq

From this definition it follows that $B(\nabla_0^{T D^{m}_{l}})$
vanishes identically, since $\mathcal{S}_0$ does. It remains to
evaluate $B(\nabla_1^{TD^{m}_{l}})$. For, note that by equation
(1.16) of \cite{BM} (or by direct calculation, since the curvature
of the disc is null)
\[
\mathcal{R}=-2 \mathcal{S}_1^2.
\]

Therefore, equation (\ref{ebm1}) gives
\begin{align*}
B(\nabla_1^{T D^{m}_{l}})&=\frac{1}{2}\int_0^1\int^B \e^{(1-u^2)
\mathcal{S}_1^2}\sum_{k=1}^\infty
\frac{1}{\Gamma\left(\frac{k}{2}+1\right)}u^{k-1} \mathcal{S}_1^k du\\
&=\frac{1}{2}\int^B \sum_{j=0,k=1}^\infty
\frac{1}{j!\Gamma\left(\frac{k}{2}+1\right)}\int_0^1 (1-u^2)^ju^{k-1} d u \mathcal{S}_1^{k+2j} \\
&=\frac{1}{2} \sum_{j=0,k=1}^\infty
\frac{1}{k\Gamma\left(\frac{k}{2}+j+1\right)}  \int^B
\mathcal{S}_1^{k+2j}.
\end{align*}

Since the Berezin integral vanishes identically whenever
$k+2j\not=m-1$, we obtain 

\beq\label{epe1} 
B(\nabla_1^{T D^{m}_{l}})=\frac{1}{2\Gamma\left(\frac{m+1}{2}\right)}
\sum_{j=0}^{\left[\frac{m}{2}-1\right]} \frac{1}{m-2j-1}  \int^B
\mathcal{S}_1^{m-1}. 
\eeq

Now consider the two cases of even and odd $m$ independently. First, assume $m=2p+1$ ($p\geq 0$). Then, using equation (\ref{pfpf}), equation (\ref{epe1}) gives
\begin{align*}
B(\nabla_1^{T D^{2p+1}_{l}})&=\frac{1}{2p!} \sum_{j=0}^{\left[p-\frac{1}{2}\right]}
\frac{1}{2p-2j}  \int^B \mathcal{S}_1^{2p}\\
&=\frac{1}{4} \sum_{n=1}^p \frac{1}{n}  \int^B \e^{-\frac{\hat \Theta}{2}}\\
&=\frac{1}{4} \sum_{n=1}^p \frac{1}{n}  Pf\left(\frac{\Theta}{2\pi}\right)\\
&=\frac{1}{4} \sum_{n=1}^p \frac{1}{n}  e(S^{2p},g_E),\\
\end{align*}
where $e(S^{2p},g_E)$ is the Euler class of $(S^{2p},g_E)$, and we use the fact that
\[
e(S^{2p}_l,g_l)= Pf\left(\frac{\Theta}{2\pi}\right)=\int^{B} \exp(-\frac{\dot\Theta}{2}).
\]

Therefore,
\begin{align*}
\log \frac{T_{\rm abs}((D^{m}_{l},g_1);\rho)}{T_{\rm abs}((D^{m}_{l},g_0);\rho)}&=\frac{1}{2}{\rm rk}(\rho)\int_{S^{m-1}_l} B(\nabla_1^{T D^{m}_{l}})\\
&=\frac{1}{8} {\rm rk}(\rho)\sum_{n=1}^p \frac{1}{n}\int_{S^{2p}_l}  e(S^{2p},g_E) \\
&=\frac{1}{8} {\rm rk}(\rho)\sum_{n=1}^p \frac{1}{n}\chi(S^{2p},g_E). \\
\end{align*}

Second, assume $m=2p$ ($p\geq 1$). Then, using equation
(\ref{pfpf}), equation (\ref{epe1}) gives
\begin{align*}
B(\nabla_1^{T D^{2p}_{l}})&=\frac{1}{2\Gamma(p+\frac{1}{2})}
\sum_{j=0}^{p-1}
\frac{1}{2p-2j-1}  \int^B \mathcal{S}_1^{2p-1}\\
&=\frac{2^p}{2\sqrt{\pi}(2p-1)!!} \sum_{j=0}^{p-1}
\frac{1}{2(p-j)-1}  \int^B \mathcal{S}_1^{2p-1}\\
&=\frac{2^{p-1}}{\sqrt{\pi}(2p-1)!!} \sum_{j=1}^{p} \frac{1}{2j-1}
\int^B \mathcal{S}_1^{2p-1},
\end{align*}
and substitution in equation (\ref{bat}) gives the formula stated in the Lemma.


\end{proof}

Next, we evaluate the integral appearing in the second equation in Lemma \ref{ddff}.

\begin{lem}\label{df} We have
\[
\frac{2^{p-1}}{\sqrt{\pi}(2p-1)!!}\int_{S^{2p-1}_l} \int^B \mathcal{S}_1^{2p-1} =p.
\]
\end{lem}

\begin{proof} First, we determine the connection one forms for the metric $g_1$ and $g_0$. We define the Christoffel symbols accordingly to
\[
\nabla_{e_{\alpha}} e_{\beta} = \tensor{\Gamma}{ _{\alpha \beta}^{\gamma}} e_{\gamma},
\]
where $\{e_\alpha\}$ is an orthonormal base, and we use the formula
\beq\label{f1}
\tensor{\Gamma}{ _{\alpha \beta}^{\gamma}} = \frac{\tensor{c}{_{\alpha
\beta} ^{\gamma}} + \tensor{c}{_{\gamma\alpha } ^{\beta}} + \tensor{c}{_{\gamma\beta} ^{\alpha}}}{2},
\eeq
where the Cartan structure constant are defined by $[e_{\alpha},e_{\beta}] = \tensor{c}{_{\alpha \beta} ^{\gamma}} e_{\gamma}$.
The orthonormal base and its dual with respect to $g_1$ are:
\begin{align*}
     {e}_r &= \frac{\partial}{\partial r},& {e}^r &= dr, \\
     {e}_{\theta_1} &= (r \prod^{n}_{j=2} \sin{\theta_j})^{-1}\frac{\partial}{\partial{\theta_1}}, & {e}^{\theta_1} &= r \prod^{n}_{j=2} \sin{\theta_j}d\theta_1, \\
     &\vdots & &\vdots \\
     {e}_{\theta_{n-1}} &= (r \sin{\theta_n})^{-1}\frac{\partial}{\partial{\theta_{n-1}}} & {e}^{\theta_{n-1}}, &= r \sin{\theta_n}d\theta_{n-1}, \\
     {e}_{\theta_n} &= \frac{1}{r}\frac{\partial}{\partial{\theta_n}}, & {e}^{\theta_n} &= rd\theta_n.
\end{align*}
%

This gives $\tensor{c}{_{\alpha\beta} ^{\gamma}} = -\tensor{c}{_{\beta\alpha}
^{\gamma}}$, $\tensor{c}{_{\alpha\alpha} ^{\gamma}} = 0, \forall \alpha,\gamma$, and the non zero are:
$\tensor{c}{_{\theta_i r} ^{\theta_i}} = r^{-1}$, and if $k>i$, $\tensor{c}{_{\theta_i \theta_k} ^{\theta_i}} =
\frac{\cos{\theta_k}}{r  \prod^{n}_{j=k } \sin{\theta_j}}$.
Using equation (\ref{f1}), the non zero Christoffel symbols are:
 $\tensor{\Gamma}{ _{\theta_i \theta_i}^{r}} = -\frac{1}{r}$,
 $\tensor{\Gamma}{ _{\theta_i r}^{\theta_i}} = \frac{1}{r}$,
 $\tensor{\Gamma}{ _{\theta_i\theta_s}^{\theta_i}} = \frac{\cos{\theta_s}}{r \prod^{n}_{j=s} \sin{\theta_j}}$, with $s>i$, and $\tensor{\Gamma}{ _{\theta_s \theta_s}^{\theta_i}} = \frac{-\cos{\theta_i}}{r \prod^{n}_{j=i} \sin{\theta_j}}$, with $i>s$.

The connection one form is the matrix $\omega_1 = \tensor{\Gamma}{ _{\gamma \beta}^{\alpha}} e^{\gamma}$,
with non zero entries
\[
\tensor{(\omega_1)}{^{\theta_i}_{\theta_k}} = \frac{\cos{\theta_k}}{r \prod^{n}_{j=k} \sin{\theta_j}}
e^{\theta_i},~i<k,\qquad \tensor{(\omega_1)}{^{r}_{\theta_i}} = -\frac{1}{r} e^{\theta_i}.
\]

The orthonormal base and its dual with respect to $g_0$ are:
\begin{align*}
     {e}_r &= \frac{\partial}{\partial r}& {e}^r, &= dr, \\
     {e}_{\theta_1} &= (l  \prod^{n}_{j=2} \sin{\theta_j})^{-1}\frac{\partial}{\partial{\theta_1}} & {e}^{\theta_1} &= (l  \prod^{n}_{j=2} \sin{\theta_j})d\theta_1 \\
     &\vdots & &\vdots \\
     {e}_{\theta_{n-1}} &= (l  \sin{\theta_n})^{-1}\frac{\partial}{\partial{\theta_{n-1}}} & {e}^{\theta_{n-1}} &= (l  \sin{\theta_n})d\theta_{n-1} \\
     {e}_{\theta_n} &= \frac{1}{l}\frac{\partial}{\partial{\theta_n}} & {e}^{\theta_n} &= l d\theta_n.
\end{align*}

The non zero Cartan constants are: $\tensor{c}{_{\theta_i \theta_k} ^{\theta_i}} =\frac{\cos{\theta_k}}{la \prod^{n}_{j=k } \sin{\theta_j}}$, with $k>i$.
Using equation (\ref{f1}), the non zero Christoffel symbols are: $\tensor{\Gamma}{ _{\theta_i \theta_s}^{\theta_i}} =
\frac{\cos{\theta_s}}{l \prod^{n}_{j=s} \sin{\theta_j}}$, when $s>i$, and $ \tensor{\Gamma}{ _{\theta_s
\theta_s}^{\theta_i}} =- \frac{\cos{\theta_i}}{l \prod^{n}_{j=i} \sin{\theta_j}}$, when $i>s$.
The non zero entries of the connection one form matrix are
\[
\tensor{(\omega_0)}{^{\theta_i}_{\theta_s}} =
\frac{\cos{\theta_s}}{l  \prod^{n}_{j=s} \sin{\theta_j}} {e}^{\theta_i},~ i<s.
\]

It follows that the unique non zero entries of $\omega_1-\omega_0$ are
\[
\tensor{(\omega_1 - \omega_0)}{^{r}_{\theta_i}} = -\frac{1}{r} e^{\theta_i}=-  \prod^{n}_{j=i+1} \sin{\theta_j} d\theta_i.
\]

Second, we determine the curvature two form  $\Theta$. Since $g_0$ is a product metric, $\Theta$ is the restriction of $\Omega_0$, and hence we compute $\Omega_0 = d\omega_0 +
\omega_0 \wedge \omega_0$. We write $\omega_0$ in the coordinate
base
\begin{align*}
    \tensor{(\omega_0)}{^{r}_{\theta_i}} &= 0,&\\
    \tensor{(\omega_0)}{^{\theta_i}_{\theta_s}} &=
\cos{\theta_s} \prod_{j=i+1}^{s-1} \sin{\theta_j} d\theta_i, & i<s,
\end{align*}
and hence $d\omega_0$ is
\begin{align*}
    \tensor{(d\omega_0)}{^{r}_{\theta_i}} =& 0,&i\leq n,\\
    \tensor{(d\omega_0)}{^{\theta_i}_{\theta_k}} =&
\prod_{j=i+1}^{k} \sin{\theta_j} d\theta_i \wedge d\theta_k \\
&- \sum_{s=i+1}^{k-1} \cos{\theta_k} \cos{\theta_s}
\prod^{k-1}_{j=i+1 , j\neq s} \sin{\theta_j} d\theta_i \wedge \theta_s, & i<k,
\end{align*}
and $\omega_0 \wedge \omega_0$ is
\begin{align*}
    \tensor{(\omega_0 \wedge \omega_0)}{^{\alpha}_{\alpha}} =& 0,\\
    \tensor{(\omega_0 \wedge \omega_0)}{^{r}_{\theta_i}}=&0,\\
 \tensor{(\omega_0 \wedge \omega_0)}{^{\theta_i}_{\theta_k}}=&\sum_{s=i+1}^{k-1} \cos{\theta_s} \cos{\theta_k}
    \prod_{j=i+1 , j\neq s}^{k-1} \sin{\theta_j} d\theta_{i} \wedge d\theta_s \\
&+ \prod_{j=i+1}^{k} \sin{\theta_j} \left(\prod_{s=k+1}^{n} \sin^{2}{\theta_s} -1 \right)d\theta_i\wedge d\theta_k, &i<k.
\end{align*}

Then, the curvature two form  $\Omega_0$ is
\begin{align*}
  \tensor{(\Omega_0)}{^{r}_{\theta_i}} =& 0,&\\
 \tensor{(\Omega_0)}{^{\theta_i}_{\theta_k}}=& \prod_{j=i+1}^{k} \sin{\theta_j}
\prod_{s=k+1}^{n} \sin^{2}{\theta_s} d\theta_i \wedge d\theta_k, &i<k,
\end{align*}
and consequently $\Theta=i^*\Omega_0$ (where $i$ denotes the inclusion of the boundary) is
\[
 \tensor{\Theta}{^{\theta_i}_{\theta_k}}= \prod_{j=i+1}^{k} \sin{\theta_j}
\prod_{s=k+1}^{n} \sin^{2}{\theta_s} d\theta_i \wedge d\theta_k, \hspace{20pt}i<k.
\]

Third, recalling that $\mathcal{S}_1^2=-\frac{1}{2}\mathcal{R}$,
\[
\int^B\mathcal{S}_1^{2p-1} = \frac{(-1)^{p-1}}{2^{p-1}}\int^B
\mathcal{S}_1 \mathcal{R}^{p-1},
\]
and using the definitions in equation (\ref{pippo}) 
\begin{align*}\hspace{0pt}\int^B&\mathcal{S}_1^{2p-1} 
 = \frac{(-1)^{p-1}}{2^{2p-1}}\int^B \left(\sum_{k=1}^{2p-1}\tensor{(\omega_1-\omega_0)}{^{r}_{\theta_k}} {\hat e^{\theta_k}}\right)
 \left(\sum_{i,j=1}^{2p-1}\tensor{(\Omega_0)}{^{\theta_i}_{\theta_j}}  {\hat e^{\theta_i}}\wedge  {\hat e^{\theta_j}}\right)^{p-1}
  \\
& \hspace{-6pt}= \frac{(-1)^{p-1}}{2^{p-1}2^p} c_B\hspace{-3pt}\left( \sum_{\substack{\sigma \in S_{2p}\\ \sigma(2\kappa-1) = 1}} {\rm sgn}(\sigma)
\tensor{(\omega_1-\omega_0)}{^{1}_{\sigma(2\kappa)}}
\tensor{(\Omega_0)}{^{\sigma(1)}_{\sigma(2)}} \ldots
\tensor{(\Omega_0)}{^{\sigma(2p-1)}_{\sigma(2p)}}\right)\hspace{-3pt},
\end{align*}
where $\kappa$ (depending on $\sigma$) is such that
$\sigma(2\kappa -1)=1$ and
$c_B=\frac{(-1)^{p(2p-1)}}{\pi^{\frac{2p+1}{2}}}$.

Observe that $\tensor{(\omega_1 -
\omega_0)}{^{1}_{\sigma(2\kappa)}}$ is a  1-form multiple of
$d\theta_{\sigma(2\kappa)-1}$ and $\tensor{(\Omega_0)}{^{i}_{j}}$ is
a 2-form multiple of $d_{\theta_{i-1}} \wedge d_{\theta_{j-1}}$. We
can twist all the 2-forms $d_{\theta_{i-1}} \wedge
d_{\theta_{j-1}}$, with $i>j$ in each term appearing in the last
line of the equation above, as the matrix is skew-symmetric. Then,
we can order the base element, in such a way that the top form
appears in each term. This produces a sign coinciding with  ${\rm
sgn}(\sigma)$. Moreover, since the matrix of the curvature two form
is skew-symmetric, the generic term in the last line of the above equation can be written in
the following form
\begin{align*}
\tensor{[{\omega_1-\omega_0}]}{^{1}_{\sigma(2\kappa)}}\tensor{
[{\Omega_0}]}{^{\sigma(1)}_{\sigma(2)}} \ldots
\tensor{[{\Omega_0}]}{^{\sigma(2p-1)}_{\sigma(2p)}} d\theta_1\wedge
\ldots \wedge d\theta_{2p-1},
\end{align*}
where $\tensor{[\xi]}{^{i}_{j}}$ denotes the coefficient of the form $\tensor{(\xi)}{^{i}_{j}}$, and $\sigma\in S_{2p}$ is such that
$\sigma(2\kappa-1) = 1$ and $\sigma(2s-1)<\sigma(2s)$ for all $s$.

We prove that
\[
\tensor{[{\omega_1-\omega_0}]}{^{1}_{\sigma(2\kappa)}}\tensor{ [{\Omega_0}]}{^{\sigma(1)}_{\sigma(2)}} \ldots
\tensor{[{\Omega_0}]}{^{\sigma(2p-1)}_{\sigma(2p)}}= -  \prod^{2p}_{i=2} (\sin{\theta_{\sigma(i)}})^{\sigma(i)-1},
\]
where $\sin\theta_{2p}=1$. The proof is by induction. If  $p=1$ the
equality holds trivially. Suppose it is true for  $p-1$. By
hypothesis, if $\tau \in S_{2p-2}$ with $\tau(1) = 1$,  then
\begin{equation*}\tensor{[{\omega_1-\omega_0}]}{^{1}_{\tau(2)}}\tensor{ [{\Omega_0}]}{^{\tau(3)}_{\tau(4)}} \ldots \tensor{[{\Omega_0}]}{^{\tau(2p-3)}_{\tau(2p-2)}}=
- \prod^{2p-2}_{i=2} (\sin{\theta_{\tau(i)}})^{\tau(i)-1}.\end{equation*}

Take $\sigma \in S_{2p}$ with $\sigma(1) = 1$. It is clear that there are
$k_0,k_1,k_2$ such that $\sigma(k_0) = 2p-2$, $\sigma(k_1) = 2p-1$ and $\sigma(k_2) = 2p$. Factoring  $\sin\theta_{\sigma(k_i)}$,
$i=0,1,2$, we obtain
\begin{align*}
\tensor{[{\omega_1-\omega_0}]}{^{1}_{\sigma(2)}}\tensor{[{\Omega_0}]}{^{\sigma(3)}_{\sigma(4)}} \ldots&
\tensor{[{\Omega_0}]}{^{\sigma(2p-1)}_{\sigma(2p)}}\\
&=(\sin\theta_{2p-2})^{2p-3}(\sin\theta_{2p-1})^{2p-2} \times
\mbox{factor}
\end{align*}
where `{\rm factor}' is a product of $\sin\theta_{j}$, $j\neq \sigma(k_0),\sigma(k_1),\sigma(k_2)$.
In this way we can rewrite `{\rm factor}' indexing it by a permutation $\tau\in S_{2p-2}$ such that the induction hypothesis holds. Then,
\[
\tensor{[{\omega_1-\omega_0}]}{^{1}_{\sigma(2)}}\tensor{[{\Omega_0}]}{^{\sigma(3)}_{\sigma(4)}} \ldots
\tensor{[{\Omega_0}]}{^{\sigma(2p-1)}_{\sigma(2p)}}= -\prod_{j=2}^{2p} (\sin\theta_{\sigma(j)})^{\sigma(j)-1}.
\]

The proof for the case $\sigma(2\kappa -1)=1$, $1<\kappa\leq p$, is similar.

We have proved that
\begin{align*}
\int^B \mathcal{S}_1^{2p-1} &= c_B \frac{(-1)^{p}p(2p-1)! }{2^{p-1}2^p} \prod_{j=2}^{2p-1} (\sin\theta_{j})^{j-1}
d\theta_1 \wedge \ldots \wedge d\theta_{2p-1}.
\end{align*}

Then
\begin{align*}
\frac{2^{p-1}}{\sqrt{\pi}(2p-1)!!}\int_{S^{2p-1}_l}\int^B \mathcal{S}_1^{2p-1} &= c_B\frac{2^{p-1}}{\sqrt{\pi}(2p-1)!!}
\frac{(-1)^{p}p(2p-1)! }{2^{p-1}2^pl^{2p-1}} \Vol(S^{2p-1}_l) \\
&=  \frac{p(2p-1)!\sqrt{\pi}}{2^{p-1}(p-1)!\sqrt{\pi}(2p-1)!!}.
\end{align*}

It is easy to see that 
\[
\frac{(2p-1)!}{(p-1)!(2p-1)!!}=2^{p-1},
\]
and the thesis follows.

\end{proof}

\begin{rem} \label{rema} In the case $m=2$, namely the 2 dimensional disc $D_l^2$, the proof of Lemma \ref{ddff} extends to the case of the cone $C_\alpha S^1_{l\sin\alpha}$. For, the curvature of the cone vanishes identically in this case. Therefore we have that
\[
\log \frac{T_{\rm abs}((C_\alpha S^1_{l\sin\alpha},g_E);\rho)}{\tau_{\rm
RS}((C_\alpha S^1_{l\sin\alpha},g_E);\rho)} =\frac{1}{2}{\rm
rk}(\rho)\frac{1}{\sqrt{\pi}} \int^B \mathcal{S}_1.
\]

The integral can be evaluated proceeding as in the proof of Lemma \ref{df}, and we obtain
\[
\log \frac{T_{\rm abs}((C_\alpha S^1_{l\sin\alpha},g_E);\rho)}{\tau_{\rm
RS}((C_\alpha S^1_{l\sin\alpha},g_E);\rho)} =\frac{1}{2}{\rm
rk}(\rho)\sin\alpha.
\]

\end{rem}

We conclude this section by computing the anomaly boundary term using the formula given in Theorem 1 of \cite{DF}.
In the even dimensional case we give the result for the more general case of the cone over the sphere. We need some more notation. Consider the homotopy $\omega_t=\omega_0+t(\omega_1-\omega_0)$, and let $\Omega_t=d\omega_t+\omega_t\wedge \omega_t$ be the corresponding curvature two form. The Chern-Simons class associated to the Euler class of $\Omega_t$  will be denoted by  $\tilde e (g_0,g_1)$, and satisfies $d \tilde e (g_0,g_1) = e(g_1) - e(g_0)$,
where $e(g_{j})$ is the Euler class of $\Omega_{j}$. Then, it is easy to see that Theorem 1 of \cite{DF} gives the following formulas:
\begin{align}
\label{ddff2}\log \frac{T_{\rm abs}((D^{2p-1}_l,g_E);\rho)}{\tau_{\rm RS}((D^{2p-1}_l,g_E);\rho)}
&=\frac{1}{4}{\rm rank}(\rho)\chi(S^{2p-2}_l,g_E)\log 2,\\
\label{ddff1}\log \frac{T_{\rm abs}((C_{\alpha} S^{2p-1}_{l\sin\alpha},g_E);\rho)}{\tau_{\rm RS}((C_{\alpha} S^{2p-1}_{l\sin\alpha},g_E);\rho)}
&=\frac{1}{2}{\rm rank}(\rho)\hspace{-2pt}\int_{S^{2p-1}_{l\sin\alpha}}i^* \tilde e(g_0,g_E),
\end{align}
where $i$ denotes the inclusion of the boundary. Proceeding in very similar way as in the proof of Lemma  \ref{df}, we compute the integral appearing in equation (\ref{ddff1}). We obtain
\[
\int_{S^n_{l\sin\alpha}}i^* \tilde e(g_0,g_E)
= \frac{(-1)^{p+1} (2p)! {\rm Vol}(S^{2p-1}_{l\sin\alpha})}{( 4\pi)^pl^{2p-1} p! (\sin\alpha)^{2p-2}   }\sum^{p-1}_{k=0} (-1)^k
\frac{(\sin\alpha)^{2k}}{2k+1}\binom{p-1}{k},
\]
and in particular for the disc ($\sin\alpha=1$)
\[
\int_{S^n_{l\sin\alpha}}i^* \tilde e(g_0,g_E)=(-1)^{p+1}.
\]

\begin{rem}\label{rema1} By comparing the approach of \cite{BM} with the approach of \cite{DF}, it turns out that in \cite{BM}  the boundary term is obtained considering an homotopy of the two metrics $g_1$ and $g_0$, namely a deformation of the metric $g_1$ into the product metric $g_0$. In \cite{DF}, the boundary term is obtained by considering an homotopy of the connections $\omega_1$ and $\omega_0$. This produces a "smaller"  term , that does not capture all the contribution of the boundary (at least in the even dimensional case). 
\end{rem}

\section{The analytic torsion of $C_\alpha S^1_{l\sin\alpha}$ and $C_\alpha S^2_{l\sin\alpha}$}
\label{s4}

In this section we compute the analytic torsion of the cones $C_\alpha S^1_{l\sin\alpha}$ and $C_\alpha S^2_{l\sin\alpha}$ by using the definition given in  equation (\ref{analytic}). For we need first the explicit knowledge of the spectrum of the Laplace operators on forms on these singular spaces, and second a suitable representation for the analytic extension of the associated zeta function, that allows to evaluate the derivative at zero. The first aspect of the problem was originally addressed by Cheeger in \cite{Che2} (see also 
\cite{Spr3} and \cite{Spr6}). In the work of Cheeger, the Hodge-de Rham theory is developed for spaces with singularity of conical type. In particular, it is proved that the Laplacian on forms is a non negative self adjoint operator on the space of square integrable forms on the cone, if some set of appropriate boundary conditions at the tip of the cone are used. We recall this point briefly in the following Remark \ref{SL}.  We give the spectrum of $\Delta^{(q)}$ on $C_\alpha S^1_{l\sin\alpha}$ and on $C_\alpha S^2_{l\sin\alpha}$ in Lemma \ref{eig1} and Lemma \ref{eig2} below, respectively. Next, to deal with the second aspect, namely an analytic extension of the zeta functions and a method to evaluating the derivative at zero, we use a method introduced by Spreafico  to deal with the zeta invariants of an abstract  class of double zeta functions.  In fact, the  eigenvalues of $\Delta^{(q)}_{C_\alpha S^n_{l\sin\alpha}}$ can be identified with the zero $z_{\nu,k}$ of some combination  of Bessel functions and their derivatives, and be enumerated with two positive indices as
$\lambda_{n,k}^{(q)}=z^2_{u_n,k}$, where the $u_n$ depends on the eigenvalues of the the Laplacian on some space of $q$-forms on the section of the cone. Using classical estimates for the zeros of the Bessel functions it is possible to prove that the relevant sequences $U$ and $S$ are contained in the class of abstract sequences introduced in \cite{Spr4} \cite{Spr9}.  This means that we can use the method of \cite{Spr3} \cite{Spr5} \cite{Spr6} \cite{Spr9}, to evaluate the derivative at zero of the associated zeta functions. We will give in the following Section \ref{s4.a} a quick overview of this method, with the main results, formulated in a way more suitable for the applications we have in mind in the present work.
Beside the approach was originally introduced in \cite{Spr3} and consequently refined and generalized in other works with different application \cite{Spr5} \cite{Spr6}, we will use \cite{Spr9} as reference for the notation and the definition.

\subsection{Spectrum of the Laplacian on forms}
\label{s4.1}

In this section we compute the spectrum of the Laplacian on forms. The general form of the solutions of the eigenvalues equation are given in \cite{Che2} and \cite{Che3}. However, we present here the explicit form of the solutions in the case under study and some details on the calculation, that we were not able to find elsewhere. Furthermore we give, in the course of the proofs,  the complete set of the eigenforms of the Laplace operator. We give a more detailed proof for the case of the circle. We denote by $\{k:\lambda\}$ the set of eigenvalues $\lambda$
with multiplicity $k$.

\begin{rem}\label{SL} Decomposing with respect to the projections on the eigenspaces of the restriction of the Laplacian on the section of the cone (i.e with respect to the angular momenta), the definition of an appropriate self adjoint extension of the Laplace operator (on functions) on a cone reduces to the analysis of the boundary values of a singular Sturm Liouville ordinary second order differential equation. The problem was addressed already by Rellich in \cite{Rel}, who parameterized the self adjoint extensions. In particular, it turns out that the there are not boundary values for the non zero mode of the angular momentum, while a boundary condition is necessary for the zero modes, and the unique self adjoint extension defined by this boundary condition is the maximal extension, corresponding to the Friedrich extension (see \cite{BS2} or \cite{Che2} for the boundary condition). The same argument  works for the Laplacian on forms. However, in the present situation we do not actually need boundary values for forms of positive degree, since the middle homology of the section of the cone is trivial (compare with \cite{Che0}).
\end{rem}

\begin{lem}\label{eig1} The spectrum of the (Friedrich extension of the) Laplacian operator $\Delta_{C_\alpha S^1_{l\sin\alpha}}^{(q)}$ on $q$-forms with absolute boundary conditions is (where $\nu={\rm cosec}\alpha$):
\begin{align*}
\Sp \Delta_{C_\alpha S^1_{l\sin\alpha}}^{(0)}=& \left\{j_{1,k}^2/l^{2}\right\}_{k=1}^{\infty}\cup \left\{2:(j_{\nu n,k}')^2/l^{2}\right\}_{n,k=1}^\infty, \\
\Sp \Delta_{C_\alpha S^1_{l\sin\alpha}}^{(1)}=& \left\{j_{0,k}^2/l^{2}\right\}_{k=1}^{\infty}\cup\left\{j_{1,k}^2/l^{2}\right\}_{k=1}^\infty\cup
\left\{2:j_{\nu n,k}^2/l^{2}\right\}_{n,k=1}^\infty \\ & \cup \left\{2:(j_{\nu n,k}')^2/l^{2}\right\}_{n,k=1}^\infty , \\
\Sp \Delta_{C_\alpha S^1_{l\sin\alpha}}^{(2)}=& \left\{j_{0,k}^2/l^{2}\right\}_{k=1}^\infty\cup \left\{2:j_{\nu n,k}^2/l^{2}\right\}_{n,k=1}^\infty. \\
\end{align*}

The spectrum of the Laplacian operator $\Delta_{C_\alpha S^1_{l\sin\alpha}}^{(q)}$ on $q$-forms with relative boundary conditions is:
\begin{align*}
\Sp \Delta_{C_\alpha S^1_{l\sin\alpha}}^{(0)}=& \left\{j_{0,k}^2/l^{2}\right\}_{k=1}^\infty\cup \left\{2:j_{\nu n,k}^2/l^{2}\right\}_{n,k=1}^\infty, \\
\Sp \Delta_{C_\alpha S^1_{l\sin\alpha}}^{(1)}=& \left\{j_{0,k}^2/l^{2}\right\}_{k=1}^{\infty}\cup\left\{j_{1,k}^2\right/l^{2}\}_{k=1}^\infty\cup
\left\{2:j_{\nu n,k}^2/l^{2}\right\}_{n,k=1}^\infty \\ & \cup \left\{2:(j_{\nu n,k}')^2/l^{2}\right\}_{n,k=1}^\infty , \\
\Sp \Delta_{C_\alpha S^1_{l\sin\alpha}}^{(2)}=& \left\{j_{1,k}^2/l^{2}\right\}_{k=1}^{\infty}\cup \left\{2:(j_{\nu n,k}')^2/l^{2}\right\}_{n,k=1}^\infty. \\
\end{align*}
\end{lem}

\begin{proof} Recall we parameterize $C_\alpha S^1_{l\sin\alpha}$ by
\[
C_\alpha S^1_{l\sin\alpha}=\left\{\begin{array}{l} x_1=x\sin\alpha\cos\theta\\x_2=x\sin\alpha\sin\theta\\x_3=x\cos\alpha\end{array}
\right. ,
\]
where $(x,\theta)\in [0,l]\times [0,2\pi]$, $l$ and $\alpha$ are fixed positive real numbers and $0<a=\sin\alpha\leq 1$. The
induced metric is
\begin{align*}
g&=d x\otimes d x+a^2 x^2 d\theta\otimes d\theta,\\
\end{align*}
and the Hodge operator is
\begin{align*}
*&:1\mapsto a x d x\wedge d\theta;\\
*&:d x\mapsto a x  d\theta,\qquad
*:d\theta \mapsto -\frac{1}{a x}  d x;\\
*&:d x\wedge d\theta\mapsto \frac{1}{ax}.
\end{align*}

The Laplacian on forms is
\begin{align*}
\Delta^{(0)}(f)=&-\b_x^2 f-\frac{1}{x}\b_x
f-\frac{1}{a^2x^2}\b_\theta^2 f;\\
\Delta^{(1)}(f_x d x+ f_\theta d_\theta)=& \left(-\b_x^2 f_x -\frac{1}{a^2 x^2} \b^2_\theta f_x +\frac{1}{x^2}
f_x-\frac{1}{x}\b_x f_x + \frac{2}{a^2
x^3}\b_\theta f_\theta \right)d x\\
&+\left(-\b^2_x f_\theta -\frac{1}{a^2 x^2}\b_\theta^2 f_\theta +\frac{1}{x}\b_x f_\theta -\frac{2}{x}\b_\theta
f_x\right)
d\theta,\\
\Delta^{(2)}(f dx\wedge d\theta)=&-\b_x^2 f+\frac{1}{x}\b_x f -\frac{2}{x^2} f -\frac{1}{a^2x^2}\b_\theta^2 f.
\end{align*}

Using the decomposition described in equation (\ref{dec}), we obtain from equations (\ref{abs}), and (\ref{rel}), the
following sets of boundary conditions. For the $0$-forms:
\beq\label{absrel0}
{\rm rel.}:\hspace{10pt}\omega(l,\theta)=0,\quad \quad{\rm abs.}:\hspace{10pt}(\partial_x
\omega)(l,\theta)=0,
\eeq
and relative BC coincide with Dirichlet BC.
For  $2$-forms
\beq\label{absrel2} {\rm abs.}:\hspace{10pt}\omega(l,\theta)=0, \quad \quad{\rm abs.}:\hspace{10pt}\left(\D_x\frac{\omega}{x}\right)(l,\theta)=0,
\eeq
and absolute BC coincide with Dirichlet BC. For  $1$-forms:
\beq\label{absrel1}
{\rm abs.}:\hspace{10pt}
\left\{\begin{array}{l}\omega_x(l,\theta)=0,\\(\D_x\omega_\theta)(l,\theta)=0,\end{array}\right.
\quad
{\rm rel.}:\hspace{10pt}
\left\{\begin{array}{l}\omega_\theta(l,\theta)=0,\\\left(\D_x(ax\omega_x)+\frac{1}{ax}\D_\theta
\omega_\theta\right)(l,\theta)=0.\end{array}\right. \eeq

Next, we solve the eigenvalues equations. Note that the Laplacian on 2-forms coincides with the one on 0-forms up to a Liouville
transform $f=x h$. Consider the eigenvalues equation for the Laplacian on 0-forms \beq\label{e0} \Delta^{(0)}
f=\left(-\D_x^2 -\frac{1}{x}\D_x -\frac{1}{a^2x^2}\D_\theta^2\right) f=\lambda^2 f. \eeq

We can decompose the problem in the eigenspaces of $-\D_\theta$. In fact, $\phi_n(\theta)=\e^{in\theta}$ is a complete
system of eigenfunctions for $-d_\theta$ on the circle $S^1$, and the eigenvalue of $\phi_n$ is $\lambda_n=n^2$, $n\in\Z$.
Thus,
\[
\Delta^{(0)}=\sum_{n\in\Z} L_n \Pi_n,
\]
where $\Pi_n$ is the projection onto the subspace generated by the eigenvector $\phi_n$ of the eigenspace relative to
the eigenvalue $\lambda_n$ and
\[
L_n=-d_x^2 -\frac{1}{x}d_x +\frac{\nu^2 n^2}{x^2},
\]
where $\nu=\frac{1}{a}$. Since $\lambda_n=\lambda_{-n}$, $\Pi_{n}=\Pi_{-n}$.  Thus $-d_\theta$ has the complete system
\[
\left\{\lambda_n=n^2; \phi_{n,+}(\theta)=\e^{i n \theta}, \phi_{n,-}(\theta)=\e^{-i n \theta}\right\}_{n\in \N},
\]
where all the eingenvalues are double up to $\lambda_0=0$ that is simple, and since also $L_n=L_{-n}$,
\[
\Delta^{(0)}=L_0\Pi_0\oplus\sum_{n=1}^\infty L_n (\Pi_{n,+}\oplus \Pi_{n,-}),
\]
where $\Pi_{\pm, n}$ is the projection on the eigenspace generated by $\phi_{\pm n}$ in the eigenspace of $\lambda_n$
(in fact the eigenspace of $\lambda_n$ is generated by the two eigenvector $\phi_{\pm n}$ for all $n\not=0$). Now, we
solve the eigenvalues equation for $L_n$ on $L^2(0,l)$, namely \beq\label{e3} L_n u=\left(-d_x^2 -\frac{1}{x}d_x
+\frac{\nu^2 n^2}{x^2}\right)u=\lambda^2 u. \eeq

This can be solved in terms of Bessel function. By classical result, equation (\ref{e3}) has
the two linearly independent solutions (assume $\mu=\nu n$ is not an integer) $y_{\pm\mu}(x)=J_{\pm |\mu|}(\lambda x)$ (where
we assume $\lambda>0$). But $J_{-|\mu|}(x)$ diverges as $x^{-|\mu|}$ at $x=0$, and therefore does not satisfy the BC at
$x=0$, or it is not in $L^2(0,l)$ (depending on the value of $\mu$). Thus in any case we only have the solution $y_+$.
This means that the eigenvalues equation (\ref{e3}) for $L_n$ has the solution $\psi_n(x)=J_{|\nu n|}(\lambda x)$, for
each $n\in \Z$; in particular it has solution $\psi_n(x)=J_{\nu n}(\lambda x)$, if $n\geq 0$, since $\nu\geq 0$.
Therefore a system of linearly independent solutions of the eigenvalues equation (\ref{e0}) for $\Delta^{(0)}$ is
\beq\label{sol0}
\begin{aligned}
&\left\{ \phi_0(\theta)\psi_0(x)=J_0(\lambda x)\right\}\\
&\cup\left\{ \phi_{n,+}(\theta)\psi_n(x)=\e^{i n\theta}J_{\nu n}(\lambda x), \phi_{n,-}(\theta)\psi_n(x)=\e^{- i
n\theta}J_{\nu n}(\lambda x)\right\}_{n\in\N_0}.
\end{aligned}
\eeq

The solution for $\Delta^{(2)}$ are given by the inverse of the above Liouville transform,
\beq\label{sol2}
\begin{aligned}
&\left\{ \phi_0(\theta)\psi_0(x)=xJ_0(\lambda x)\right\}\\
&\cup\left\{ \phi_{n,+}(\theta)\psi_n(x)=x\e^{i n\theta}J_{\nu n}(\lambda x), \phi_{n,-}(\theta)\psi_n(x)=x\e^{- i
n\theta}J_{\nu n}(\lambda x)\right\}_{n\in\N_0}.
\end{aligned}
\eeq

The eigenvalues equation for the Laplacian on $1$-forms: \beq \label{1e} \Delta^{(1)} \omega=\lambda^2\omega, \eeq with
$\omega=f_x dx+f_\theta d\theta$  corresponds to the system of partial differential equations
\beq\label{1e1}
\left\{\begin{aligned}&-\D_x^2 f_x -\frac{1}{x}\D_x f_x+\frac{-\nu^2\D^2_\theta+1}{x^2}  f_x
+ \frac{2\nu^2}{x^3}\D_\theta f_\theta=\lambda^2 f_x, \\
&-\D^2_x f_\theta +\frac{1}{x}\D_x f_\theta+\frac{-\nu^2 \D_\theta^2}{x^2} f_\theta
 -\frac{2}{x}\D_\theta f_x=\lambda^2 f_\theta.\end{aligned}\right.
\eeq

Since a base for $L^2(S^1)$ is given by the functions $\e^{i n \theta}$ with integer $n$, we consider solutions of the
type $\omega=f_x(x)\e^{im\theta}dx+f_\theta(x)\e^{i n\theta}d\theta$, with integers $m$ and $n$. Substitution in equation
(\ref{1e1}) gives
\[
\left\{
\begin{aligned}&-\D_x^2 f_x \e^{im\theta}-\frac{1}{x}\D_x f_x\e^{im\theta}+\frac{(\nu m)^2+1}{x^2}  f_x
\e^{im\theta}+ \frac{2i\nu^2 n}{x^3} f_\theta\e^{in\theta}=\lambda^2 f_x\e^{im\theta}, \\
&-\D^2_x f_\theta \e^{in\theta}+\frac{1}{x}\D_x f_\theta\e^{in\theta}+\frac{(\nu n)^2}{x^2} f_\theta\e^{in\theta}
 -\frac{2i m}{x} f_x\e^{im\theta}=\lambda^2 f_\theta\e^{in\theta},\end{aligned}\right.
\]
that is satisfied if and only if $m=n$. Therefore, it follows that the solutions of equation (\ref{1e}) are of the form
\[
\omega=\e^{in\theta}(f_x(x)dx+f_\theta(x)d\theta),
\]
with $n\in \Z$, or in other words, that the operator $\Delta^{(1)}$ decomposes as
\[
\Delta^{(1)}=\sum_{n\in\Z} L_n \Pi_n,
\]
where
\[
L_n=\left(
\begin{array}{cc}-d_x^2  -\frac{1}{x}d_x +\frac{(\nu n)^2+1}{x^2}  &
 \frac{2i\nu^2 n}{x^3} \\-\frac{2i n}{x}&
-d^2_x  +\frac{1}{x}d_x +\frac{(\nu n)^2}{x^2}   \end{array}\right),
\]
on $(L^2(0,1))^2$, and $\Pi_n$ is the projector onto the subspace generated by $\e^{i n\theta}$ of the eigenspace
relative to the eigenvalue $n^2$ of $-d^2_\theta$. Therefore we need to solve the eigenvalues equation
\[
L_n u=\lambda_n^2 u,
\]
where $u=(f_x,f_\theta)$ are two functions in $L^2(0,l)$. This corresponds
to the system
\[
\left\{\begin{aligned}&-d_x^2 f_x -\frac{1}{x}d_x f_x+\frac{(\nu n)^2+1}{x^2}  f_x
+ \frac{2i\nu^2 n}{x^3} f_\theta=\lambda_n^2 f_x, \\
&-d^2_x f_\theta +\frac{1}{x}d_x f_\theta+\frac{(\nu n)^2 }{x^2} f_\theta
 -\frac{2 in}{x} f_x=\lambda_n^2 f_\theta.\end{aligned}\right.
\]

With the change of base $(f_x,f_\theta)=(\nu g_x,-ix g_\theta)$, we obtain
\beq\label{1e3}
\left\{\begin{aligned}& \left(-d_x^2
-\frac{1}{x}d_x +\frac{(\nu n+1)^2}{x^2}\right)(g_x+ g_{\theta})
=\lambda_n^2 (g_x+ g_{\theta}), \\
& \left( -d^2_x  -\frac{1}{x}d_x +\frac{(\nu n -1)^2}{x^2} \right)(g_x- g_{\theta})
 =\lambda_n^2 (g_x - g_{\theta}).\end{aligned}\right.
\eeq

By classical results on the solution of the Bessel equation, and taking only the $L^2$ solution,
we have that a complete set of linearly independent solution is given by the two vectors
\[
\left\{\begin{aligned}
(g_x,g_\theta)_n&=(J_{|\nu n+1|}(\lambda_n x),J_{|\nu n+1|}(\lambda_n x)),\\
(g_x,g_\theta)_n&=(J_{|\nu n-1|}(\lambda_n x),-J_{|\nu n-1|}(\lambda_n x)).
\end{aligned}\right.
\]

Therefore,  the eigenvalues equation (\ref{1e}) relative to the operator $\Delta^{(1)}$, has the following complete set of linearly independent $L^2$ solutions with $n\in\Z$

\beq\label{sol1} \left\{\omega_{n,\pm}=J_{|\nu n\pm1|}(\lambda_n x)\e^{in\theta}\left(\nu dx-ixd\theta\right)
\right\}.
\eeq

Eventually, we apply the boundary conditions. For $0$-forms, decomposing as in equation (\ref{dec}) $\omega_{tg} =
\omega$ and $\omega_{\rm norm} = 0$. Relative boundary conditions given in equation (\ref{absrel0}) applied to the solutions
in equation (\ref{sol0}), give $\lambda=\frac{j_{\nu n,k}}{l}$, where $j_{\nu, k}$ are the positive zeros of the Bessel
function $J_{\nu}$, arranged in increasing order, with $k=1,2,\dots$. Since it is known that the set
$\{J_{\nu}(j_{\nu,k} x)\}_{k=1,2,\dots}$ defines an orthogonal basis of the space $L^2(0,1)$, we have proved that the
set
\begin{align*}
\left\{ \phi_0(\theta)\psi_{0,k}(x)=J_0(\frac{j_{0,k}}{l}x), \phi_{n,\pm}(\theta)\psi_{n,k}(x)=\e^{\pm i n\theta}J_{\nu
n}(\frac{j_{\nu n,k}}{l} x)\right\}_{n\in\N_0},
\end{align*}
defines a complete  set of orthogonal linear independent solutions of the eigenvalues equation (\ref{e0}) for $\Delta^{(0)}$
with Dirichlet BC at $x=l$ on $L^2(0,l)$, and where $\lambda=\frac{j_{\nu n,k}}{l}$ for both $\phi_{n,\pm}J_{\nu n}$
when $n\not=0$.
Absolute boundary conditions are given in equation (\ref{absrel0}). Applying to the solutions in equation (\ref{sol0}), we
obtain
\[
\frac{\partial \omega}{\partial x}(l,\theta) =\lambda_n e^{i n \theta} J'_{\mid\nu n\mid} (\lambda_n l) = 0,
\]
that give $\lambda_n = \frac{j'_{\mid \nu n\mid,k}}{l}$, where the $j'_{\nu n, k}$ are the zeros of $J'_{\nu n}(z)$.

The result for $2$-forms is the dual of that for $0$-forms. Note that, applying the inverse of
the previous Liouville transform, we get a complete system for $\Delta^{(2)}$ with absolute boundary conditions:
\[
\left\{ \lambda^{(2)}_{n,k}=\frac{j_{\nu|n|,k}^2}{l^2},
\omega^{(2)}_{n,k}(x,\theta)=\phi_{n}(\theta)\rho_{\nu|n|,k}(x)=\e^{i n\theta}xJ_{\nu |n|}(\frac{j_{\nu |n|,k}}{l}
x)\right\}_{n\in\Z,k\in\N_0}.
\]

For a $1$-form $\omega(x,\theta)=\omega_x(x,\theta)dx+\omega_\theta(x,\theta)d\theta$, $\omega_{\rm tan}=\omega_\theta$ and $\omega_{\rm norm}=\omega_x$. Note that none of the solutions in \eqref{sol1} satisfy the BC  \eqref{absrel1}, for $\lambda_n\neq 0$. So we consider linear combinations $\omega_{n,\pm}(x,\theta) = \omega_{1,n}(x,\theta) \pm \omega_{2,n}(x,\theta)$.
Applying the BC \eqref{absrel1} to $\omega_{n,\pm}(x,\theta)$ we obtain, if $n\neq 0$,  the eigenvalues

\begin{align*} \lambda_{n,+}^2 &= j_{|\nu n|,k}^2/l^2, & \lambda_{n,-}^2 &= (j'_{|\nu n|,k})^2/l^2.
\end{align*}

If $n=0$, we have $\lambda^2_{0,+}=j^2_{1,k}/l^2$, and $\lambda^2_{0,-}=j^2_{0,k}/l^2$.
Applying the BC \eqref{absrel1} to  $\omega_{n,\pm}(x,\theta)$ we obtain, if $n\neq 0$,  the eigenvalues
\begin{align*} \lambda_{n,+}^2 &= (j'_{|\nu n|,k})^2/l^2, & \lambda_{n,-}^2 &= j_{|\nu n|,k}^2/l^2.
\end{align*}

If $n=0$, we have $\lambda^2_{0,+}=j^2_{0,k}/l^2$, and $\lambda^2_{0,-}=j^2_{1,k}/l^2$. The eigenforms follow from equation \eqref{sol1}.

\end{proof}

\begin{lem}\label{eig2}
The spectrum of the (Friedrich extension of the) Laplacian operator $\Delta_{C_{\alpha} S^2_{l\sin\alpha}}^{(q)}$ on $q$-forms with absolute boundary conditions is:
\begin{align*}
\Sp \Delta_{C_{\alpha} S^2_{l\sin\alpha}}^{(0)}=& \left\{(2n+1): \tilde j_{\mu_n,k,-}^2/l^{2}\right\}_{n,k=1}^{\infty}
\cup \left\{j_{\frac{3}{2},k}^2/l^{2}\right\}_{k=1}^\infty, \\
\Sp \Delta_{C_{\alpha} S^2_{l\sin\alpha}}^{(1)}=& \left\{j_{\frac{3}{2},k}^2/l^{2}\right\}_{k=1}^\infty\cup
\left\{(2n+1):j_{\mu_n, k}^2/l^{2}\right\}_{n,k=1}^\infty\\
&\cup\left\{(2n+1):\tilde j_{\mu_n,k,+}^2/l^{2}\right\}_{n,k=1}^\infty\cup
\left\{(2n+1):\tilde j_{\mu_n,k,-}^2/l^{2}\right\}_{n,k=1}^\infty,\\
\Sp \Delta_{C_{\alpha} S^2_{l\sin\alpha}}^{(2)}=& \left\{j_{\frac{1}{2},k}^2/l^{2}\right\}_{k=1}^\infty\cup
\left\{(2n+1):j_{\mu_n, k}^2/l^{2}\right\}_{n,k=1}^\infty\\
&\cup\left\{(2n+1):\tilde j_{\mu_n,k,+}^2/l^{2}\right\}_{n,k=1}^\infty\cup
\left\{(2n+1):j_{\mu_n,k}^2/l^{2}\right\}_{n,k=1}^\infty, \\
\Sp \Delta_{C_{\alpha} S^2_{l\sin\alpha}}^{(3)}=& \left\{(2n+1):j_{\mu_n,k}^2/l^{2}\right\}_{n,k=1}^{\infty} \cup
\left\{j_{\frac{1}{2},k}^2/l^{2}\right\}_{k=1}^\infty ,\\
\end{align*}

The spectrum of the Laplacian operator $\Delta_{C_{\alpha} S^2_{l\sin\alpha}}^{(q)}$ on $q$-forms with relative boundary conditions is:
\begin{align*}
\Sp \Delta_{C_{\alpha} S^2_{l\sin\alpha}}^{(0)}=& \left\{(2n+1):j_{\mu_n,k}^2/l^{2}\right\}_{n,k=1}^{\infty} \cup \left\{j_{\frac{1}{2},k}^2/l^{2}\right\}_{k=1}^\infty ,\\
\Sp \Delta_{C_{\alpha} S^2_{l\sin\alpha}}^{(1)}=& \left\{j_{\frac{1}{2},k}^2/l^{2}\right\}_{k=1}^\infty\cup
\left\{(2n+1):j_{\mu_n, k}^2/l^{2}\right\}_{n,k=1}^\infty\\
&\cup\left\{(2n+1):\tilde j_{\mu_n,k,+}^2/l^{2}\right\}_{n,k=1}^\infty\cup
\left\{(2n+1):j_{\mu_n,k}^2/l^{2}\right\}_{n,k=1}^\infty, \\
\Sp \Delta_{C_{\alpha} S^2_{l\sin\alpha}}^{(2)}=& \left\{j_{\frac{3}{2},k}^2/l^{2}\right\}_{k=1}^\infty\cup
\left\{(2n+1):j_{\mu_n, k}^2/l^{2}\right\}_{n,k=1}^\infty\\
&\cup\left\{(2n+1):\tilde j_{\mu_n,k,+}^2/l^{2}\right\}_{n,k=1}^\infty\cup
\left\{(2n+1):\tilde j_{\mu_n,k,-}^2/l^{2}\right\}_{n,k=1}^\infty,\\
\Sp \Delta_{C_{\alpha} S^2_{l\sin\alpha}}^{(3)}=& \left\{(2n+1):\tilde j_{\mu_n,k,-}^2/l^{2}\right\}_{n,k=1}^{\infty} \cup \left\{j_{\frac{3}{2},k}^2/l^{2}\right\}_{k=1}^\infty, \\
\end{align*}
where (recall $\nu={\rm cosec}\alpha$)
\[
\mu_n=\sqrt{\nu^2 n(n+1)+\frac{1}{4}},
\]
and where the $\tilde j_{\nu,k,\pm}$ are the zeros of the function $G^{\pm}_{\nu}(z)=\pm\frac{1}{2}J_{\nu}(z)+zJ'_\nu(z)$.
\end{lem}

\begin{proof} Recall we parameterize $C_{\alpha}S^{2}_{l\sin\alpha}$ by
\[C_{\alpha}S^2_{l\sin\alpha}=\begin{cases}
x_1=x\sin{\alpha}\sin{\theta_2}\cos{\theta_1} \\[8pt]
x_2=x\sin{\alpha}\sin{\theta_2}\sin{\theta_1} \\[8pt]
x_3=x\sin{\alpha}\cos{\theta_2} \\[8pt]
x_4=x\cos{\alpha}
\end{cases}\] where $(x,\theta_1,\theta_2)\in [0,l]\times[0,2\pi]\times[0,\pi]$, $0<\alpha\leq \pi/2$ is
a fixed real number and $0< a=\sin\alpha\leq 1$.
The induced metric is
\[g = dx\otimes dx + (a^2x^2\sin^2\theta_2) d\theta_1\otimes d\theta_1 + (a^2x^2) d\theta_2 \otimes
     d\theta_2 .\]
The Hodge star acts as follows
\begin{align*}
*&:1\mapsto a^2x^2\sin\theta_2 dx\wedge d\theta_1 \wedge d\theta_2 ;\\
*&:d x\mapsto a^2x^2\sin\theta_2 d\theta_1 \wedge d\theta_2,  *:d\theta_1 \mapsto \frac{-1}{\sin\theta_2} dx
\wedge d\theta_2,*: d\theta_2 \mapsto \sin\theta_2 dx \wedge d\theta_1;\\
*&:d x \wedge d\theta_1 \mapsto \frac{1}{\sin\theta_2} d\theta_2,  *:dx \wedge d\theta_2 \mapsto -\sin\theta_2
d\theta_1,*: d\theta_1 \wedge d\theta_2 \mapsto \frac{1}{a^2x^2\sin\theta_2} dx;\\
*&:dx \wedge d\theta_1 \wedge d\theta_2 \mapsto \frac{1}{a^2x^2\sin\theta_2};
\end{align*}

The Laplacian on forms reads
\[
 \Delta^{(0)}(\omega)=-\left( \frac{2}{x}\partial_x \omega + \partial_x^2 \omega + \frac{\cos\theta_2}{a^2x^2\sin\theta_2}\partial_{\theta_2} \omega
  + \frac{\partial_{\theta_2}^2 \omega}{a^2x^2} + \frac{\partial_{\theta_1}^2 \omega}{a^2x^2\sin^2\theta_2}\right)
\]

\begin{multline*}\Delta^{(1)}(\omega_x dx  + \omega_{\theta_1}d\theta_1 + \omega_{\theta_2} d\theta_2 )= \\[15pt]
\left( -\partial_x^2\omega_x + \frac{2}{x^2}\omega_x - \frac{2}{r}\partial_x \omega_x -
\frac{1}{a^2x^2}\partial_{\theta_2}^2\omega_x -\frac{\cos\theta_2}{a^2x^2\sin\theta_2}\partial_{\theta_2} \omega_x\right. \\
\left. - \frac{1}{a^2x^2\sin^2\theta_2}\partial_{\theta_1}^2\omega_x +\frac{2}{a^2x^3\sin^2\theta_2}\partial_{\theta_1}
\omega_{\theta_1}+
\frac{2}{a^2x^3}\partial_{\theta_2} \omega_{\theta_2}  + \frac{2\cos\theta_2}{a^2x^3\sin\theta_2}\omega_{\theta_2}   \right)dx  \\[15pt]
 +\left( -\partial_x^2\omega_{\theta_1}-\frac{2}{r}\partial_{\theta_1}\omega_x
-\frac{1}{a^2x^2}\partial_{\theta_2}^2\omega_{\theta_1}\right.\\
\left.+\frac{\cos\theta_2}{a^2x^2\sin\theta_2}(\partial_{\theta_2} \omega_{\theta_1}
-2\partial_{\theta_1} \omega_{\theta_2} )-\frac{1}{a^2x^2\sin^2\theta_2}\partial_{\theta_1} ^2\omega_{\theta_1}\right)d\theta_1   \\[15pt]
+\left( -\partial_x^2\omega_{\theta_2}  - \frac{2}{r}\partial_{\theta_2} \omega_x -
\frac{1}{a^2x^2}\partial_{\theta_2}^2\omega_{\theta_2}  - \frac{\cos\theta_2}{a^2x^2\sin\theta_2}\partial_{\theta_2} \omega_{\theta_2} \right. \\
\left. + \frac{1}{a^2x^2\sin^2\theta_2}(\omega_{\theta_2}  - \partial_{\theta_1} ^2\omega_{\theta_2} ) +
\frac{2\cos\theta_2}{a^2x^2\sin^3\theta_2}\partial_{\theta_1} \omega_{\theta_1}\right)d\theta_2 \\
\end{multline*}

Applying the decomposition described in equation (\ref{dec}), we obtain from equations (\ref{abs}), and (\ref{rel}),
the following sets of boundary conditions.

For the $0$-forms:
\beq\label{absrel0S2}
\begin{aligned}
{\rm abs. }&: \partial_x \omega(l,\theta_1,\theta_2)=0,&\quad
{\rm rel. }: \omega(l,\theta_1,\theta_2) = 0.
\end{aligned}
\eeq

For $1$-forms:
\beq\label{absrel1S2}
\begin{aligned}
{\rm abs.}&:
    \left\{\begin{array}{lll} \omega_x (l,\theta_1,\theta_2) = 0 \\
    \partial_x \omega_{\theta_1}(l,\theta_1,\theta_2)=0 \\
    \partial_x\omega_{\theta_2}(l,\theta_1,\theta_2)=0,\end{array}\right. &\quad{\rm rel. }:
    \left\{\begin{array}{lll} \omega_{\theta_1}(l,\theta_1,\theta_2) = 0 \\
    \omega_{\theta_2}(l,\theta_1,\theta_2) = 0 \\
    \partial_x(x^{2} \omega_x)(l,\theta_1,\theta_2) = 0,\end{array}\right.
\end{aligned}
\eeq

Next we solve the eigenvalue equations. For $0$-forms:
\begin{equation}\label{e0S2}
\Delta^{(0)}(\omega) = - \partial^{2}_{x} \omega - \frac{2}{x}\partial_x \omega +  \frac{1}{a^{2}x^{2}}
\Delta^{(0)}_{S^{2}}(\omega) = \lambda^{2} \omega
\end{equation}

We can decompose the problem in the eigenspaces of $\Delta^{(0)}_{S^{2}}$. Let $Y^{k}_{n}(\theta_1,\theta_2) =
\e^{ik\theta_1} P^{\mid k\mid}_{n}(\cos\theta_2)$, where $P^{\mid k\mid}_{n}(\cos\theta_2)$ are the associated Legendre
polynomials and $|k| \leq n$. $Y^{k}_{n}(\theta_1,\theta_2)$ is a complete system of eigenforms for
$\Delta^{(0)}_{S^{2}}$ and the eigenvalues are $n(n+1)$, with multiplicity $2n+1$ and $n\in \Z$, $n \geq 0$. Thus,
\[\Delta^{(0)} = \sum_{n \geq 0} T_n \Pi_n,\]
with
\[T_n = - \partial^{2}_{x} \omega - \frac{2}{x}\partial_x \omega +  \frac{\nu^2 n(n+1)}{x^{2}}, \] and the eigenvalues equation reads
\[T_n (u) = \left(- \partial^{2}_{x} \omega - \frac{2}{x}\partial_x \omega +  \frac{\nu^2 n(n+1)}{x^{2}}\right)u = \lambda^{2} u.\]

This can be solved in terms of Bessel function and the solution is $u_n (x) = x^{-\frac{1}{2}}J_{\mu_n}(\lambda_n x)$.
Hence the solution for the $0$-Laplacian equation is
\begin{equation}\label{sol0S2}
\alpha^{(0)}_{n}(x,\theta_1,\theta_2) = x^{-\frac{1}{2}}J_{\mu_n}(\lambda_n x) Y^{k}_{n}(\theta_1,\theta_2).
\end{equation}

For 1-form we have
\begin{equation}\label{1eS2}
\Delta^{(1)}(\omega) = \lambda^{2} \omega
\end{equation}
with $\omega = \omega_x dx + \omega_{\theta_1} d\theta_1 + \omega_{\theta_2} d\theta_2$. 
Write $\omega= f_{\theta_1\theta_2}(x)\phi(\theta_1,\theta_2) + f_x(x) h(\theta_1,\theta_2) dx$ where $\phi =
f_{\theta_1} d\theta_1+ f_{\theta_2}d\theta_2$ and $h(\theta_1,\theta_2)$ is a 0-form on $S^{2}$. Hence replacing in
\eqref{1eS2} we have the system
\begin{equation}\label{system1}
\left\{\!\begin{aligned} &(-d_x^2 f_{\theta_1\theta_2} )\phi +\frac{\Delta^{(1)}_{S^{2}}(\phi)}{a^{2}x^2}
f_{\theta_1\theta_2} - \frac{2 f_x d(h)}{x}=
\lambda^2 f_{\theta_1\theta_2} \phi, \\
&dx\!\! \left(\!(-d^2_x f_x -\!\frac{2}{x}d_x f_x +\! \frac{2}{x^2} f_x) h +\! \frac{\Delta^{(0)}_{S^{2}}(h)}{a^{2} x^{2}} f_x -\!
\frac{2 f_{\theta_1\theta_2} d^{\dag}_{S^{2}}(\phi)}{a^{2} x^{3}}\right) \! =\! dx (\lambda^2 f_x  h).
\end{aligned}\right.
\end{equation}
Consider $f_x =0$ or $h = 0$ and $\phi$ a coexact eigenform  on $S^{2}$ with  non zero eigenvalue. We have
the equation, for $n \geq 1$,
\[(-d_x^2 f_{\theta_1\theta_2} )\phi +\frac{\nu^{2} n(n+1) f_{\theta_1\theta_2}}{x^2}  \phi = \lambda^2 f_{\theta_1\theta_2} \phi.
\]

Solving this equation in $x$ we find $f_{\theta_1\theta_2} = x^{\frac{1}{2}}J_{\mu_n}(\lambda_n x)$ and then
\[
\alpha^{(1)}_{n} = x^{\frac{1}{2}}J_{\mu_n}(\lambda_n x) \phi.
\]

 Note that $\phi = d^{\dag}_{S^2}(\sin\theta_2
Y^{k}_{n}(\theta_1,\theta_2) d\theta_1 \wedge d\theta_2)$.
Now we consider $f_x \neq 0$, $f_{\theta_1\theta_2} \neq 0$, and $h$  a coexact 0-eigenform  of $S^2$ with non zero
eigenvalue such that $d (h) = \phi$. Hence, $\Delta^{(1)}_{S^{2}}(\phi) = n(n+1) \phi$, $d^{\dag}_{S^{2}}(d(h)) = n(n+1)
h$, and the system \eqref{system1} becomes
\begin{equation*}
\left\{  \begin{aligned} &(-d_x^2 f_{\theta_1\theta_2} )\phi +\frac{n(n+1) f_{\theta_1\theta_2}}{a^{2}x^2} \phi - \frac{2
f_{\theta_1\theta_2} \phi}{x}=
\lambda^2 f_{\theta_1\theta_2} \phi, \\
&  (-d^2_x f_x - \frac{2}{x}d_x f_x +  \frac{2}{x^2} f_x) h +\frac{n(n+1) h}{a^{2} x^{2}} f_x - \frac{2 n(n+1)
f_{\theta_1\theta_2}(\phi)}{a^{2} x^{3}} = \lambda^2 f_x  h.
\end{aligned}\right.
\end{equation*}

Changing the base by $(f_{\theta_1\theta_2},f_x) = (x^{-\frac{1}{2}} g_x , \partial_x(x^{-\frac{1}{2}} g_x ))$, we solve
the system, and the solution is $g_x = J_{\mu_n} (\lambda_n x)$. Hence the solution for the system
\eqref{system1} in this case is
\[
\beta^{(1)}_{n} = x^{-\frac{1}{2}}J_{\mu_n} (\lambda_n x) \phi(\theta_1,\theta_2) + \partial_x(x^{-\frac{1}{2}}
J_{\mu_n} (\lambda_n x) ) h(\theta_1,\theta_2) dx
\] where $h(\theta_1,\theta_2) = Y^{k}_{n}(\theta_1,\theta_2)$.
Consider now $f_x \neq 0$, $f_{\theta_1\theta_2} \neq 0$, and $\psi$ a coexact 0-eigenform  of $S^2$ with non zero
eigenvalue such that $d(\psi) = \phi$ and $d^{\dag}_{S^{2}} d (\psi) = h$. Then $h = n(n+1) \psi $, and the system
\eqref{system1} becomes
\begin{equation}\label{system3}
\left\{\begin{aligned} &(-d_x^2 f_{\theta_1\theta_2} )\phi +\frac{n(n+1)f_{\theta_1\theta_2}}{a^{2}x^2} \phi  - \frac{2
f_x n(n+1)}{x} \phi =
\lambda^2 f_{\theta_1\theta_2} \phi, \\
&dx\left((-d^2_x f_x -\frac{2}{x}d_x f_x + \frac{2}{x^2} f_x) h +\frac{n(n+1)f_x}{a^{2} x^{2}} h - \frac{2
f_{\theta_1\theta_2}}{a^{2} x^{3}} h \right) = dx (\lambda^2 f_x  h).
\end{aligned}\right.
\end{equation}

Changing the base by $(f_{\theta_1\theta_2},f_x) = (\partial_x(x^{\frac{1}{2}} g_x) , x^{-\frac{3}{2}} g_x )$ we solve
the system \eqref{system3} and the solution is $g_x = J_{\mu_n} (\lambda_n x)$. Hence the solution for the system
\eqref{system1} in this case is
\[
\gamma^{(1)}_{n} = \partial_x(x^{\frac{1}{2}} J_{\mu_n} (\lambda_n x)) (\lambda_n x) \phi(\theta_1,\theta_2) +
x^{-\frac{3}{2}} J_{\mu_n} (\lambda_n x) h(\theta_1,\theta_2) dx
\] where $\psi(\theta_1,\theta_2) = Y^{k}_{n}(\theta_1,\theta_2)$.
In the case $\Delta^{(1)}_{S^{2}}(\phi) = \Delta^{(0)}_{S^{2}}(h) = 0$ we
have the equation
\begin{equation*}
dx\left((-d^2_x f_x -\frac{2}{x}d_x f_x + \frac{2}{x^2} f_x) h \right) = dx (\lambda^2 f_x  h).
\end{equation*}

Hence changing the base by $f_x = \partial_x(x^{- \frac{1}{2}} g_x)$ we find $g_x = J_{\frac{1}{2}}(\lambda_0 x)$ and
the solution is \[ D^{(1)} = \partial_x(x^{- \frac{1}{2}} J_{\frac{1}{2}}(\lambda_0 x)) h.\]

Since we know that the Hodge decomposition of square integrable forms into exact and coexact forms holds as in the smooth case \cite{Che3}, we conclude that the equation \eqref{1eS2} has the following complete set of linearly
independent $L^2$ solutions
\begin{equation}\label{sol1S2} \begin{aligned}
\alpha^{(1)}_{n} &= x^{\frac{1}{2}} J_{ \mu_n }(\lambda x) d^{\dag}_{S^2}(\sin\theta_2
Y^{k}_{n}(\theta_1,\theta_2) d\theta_1 \wedge d\theta_2)\\
\beta^{(1)}_{n} &=\partial_x(x^{-\frac{1}{2}}J_{ \mu_n }(\lambda x)) Y^{k}_{n}(\theta_1,\theta_2) dx
    + x^{-\frac{1}{2}}J_{ \mu_n }(\lambda x) d(Y^{k}_{n}(\theta_1,\theta_2))\\
\gamma^{(1)}_{n} &= x^{-\frac{3}{2}}J_{ \mu_n }(\lambda x)(n(n+1))\nu^{2} Y^{k}_{n}(\theta_1,\theta_2) dx
    + \partial_x(x^{\frac{1}{2}}J_{ \mu_n }(\lambda x)) d(Y^{k}_{n}(\theta_1,\theta_2))\\
D^{(1)} &= \partial_x(x^{-\frac{1}{2}}J_{  \frac{1}{2} }(\lambda x)) dx
\end{aligned}
\end{equation}

Next we determine the eigenvalues.
\subsubsection*{0-forms}
We have only two type of forms in \eqref{sol0S2}, that are,
 \begin{align*}\alpha^{(0)}_{n} &= x^{-\frac{1}{2}} J_{\mu_n} (\lambda x) \phi^{(0)}_{n}, \\
  E^{(0)}_n &= x^{-\frac{1}{2}}J_{\frac{1}{2}}(\lambda x)\phi^{(0)}_{n}.\end{align*}

Using the absolute BC in \eqref{absrel0S2} we have
\begin{align*} \partial_x (\alpha^{(0)}_{1,n})(l,\theta_1,\theta_2) &=\partial_x (x^{-\frac{1}{2}} J_{\mu_n}(\lambda x)\phi^{(0)}_{n})(l,\theta_1,\theta_2) = 0 \\
\partial_x(E^{(0)}_n)(l,\theta_1,\theta_2) &= \partial_x(x^{-\frac{1}{2}}J_{\frac{1}{2}}(\lambda x)\phi^{(0)}_{n})(l,\theta_1,\theta_2)=0\end{align*}
and so we need the square of the solutions of $-\frac{1}{2}l^{-\frac{3}{2}} J_{\mu_n}(l\lambda) +
l^{-\frac{1}{2}}\lambda J'_{\mu_n}(l\lambda) = 0$ and $-\lambda^{-\frac{1}{2}}J_{\frac{3}{2}}(l\lambda)=0$ that are
\begin{equation*}\label{ev:abs0}
  \lambda^2=\frac{\tilde{j}^2_{\mu_n,k,-}}{l^2} \quad \mbox{and} \quad \lambda^2=\frac{j^2_{\frac{3}{2},k}}{l^2}.
\end{equation*}

Using the relative BC in \eqref{absrel0S2} we have $\alpha^{(0)}_{n} (l,\theta_1,\theta_2)=E^{(0)}_{n} (l,\theta_1,\theta_2) = 0$
and the eigenvalues are
\begin{equation*}\label{ev:rel0} \lambda^2 = \frac{j^2_{\mu_n,k}}{l^2}\quad \mbox{and} \quad \lambda^2=\frac{j^2_{\frac{1}{2},k}}{l^2}.\end{equation*}

\subsubsection*{1-forms}
In this case we have the four types of forms in \eqref{sol1S2} ($s=1,2$):
 \begin{align*}
    \alpha^{(1)}_{n} &= x^{\frac{1}{2}} J_{ \mu_n }(\lambda x)\phi^{(1)}_{n}, \\
    \beta^{(1)}_{n} &=\partial_x(x^{-\frac{1}{2}}J_{ \mu_n }(\lambda x)) \phi^{(0)}_{n} dx
    + x^{-\frac{1}{2}}J_{ \mu_n }(\lambda x) d\phi^{(0)}_{n},\\
    \gamma^{(1)}_{n} &= x^{-\frac{3}{2}}J_{ \mu_n }(\lambda x)(n(n+1))\nu^{2} \phi^{(0)}_{n} dx
    + \partial_x(x^{\frac{1}{2}}J_{ \mu_n }(\lambda x)) d\phi^{(0)}_{n}, \\
    D^{(1)}_n &= \partial_x(x^{-\frac{1}{2}}J_{\frac{1}{2}}(\lambda x))\phi^{(0)}_{n}dx.
\end{align*}

Using the absolute BC in \eqref{absrel1S2} we have, for the four types,
\begin{align*}
\partial_x ((\alpha^{(1)}_{n})_{\theta_s})(l,\theta_1,\theta_2) &=
    \partial_x (x^{\frac{1}{2}} J_{\mu_n}(\lambda x))(l) = 0 \\
(\beta^{(1)}_{n})_x(l,\theta_1,\theta_2)=\partial_x((\beta^{(1)}_{n})_{\theta_s})(l,\theta_1,\theta_2) &=
    \partial_x(x^{-\frac{1}{2}}J_{\mu_n}(\lambda x))(l) =0\\
(\gamma^{(1)}_{n})_x(l,\theta_1,\theta_2) &= J_{\mu_n}(l\lambda)=0 \\
\partial_x(\gamma^{(1)}_{n})_{\theta_s}(l,\theta_1,\theta_2) &=
    -\frac{1}{4}J_{\mu_n}(l\lambda)+\lambda J'_{\mu_n}(l\lambda)+\lambda^2J''_{\mu_n}(l\lambda)=0 \\
\partial_x(D^{(1)}_n) &= \partial_x(x^{-\frac{1}{2}}J_{\frac{1}{2}}(\lambda x))(1)=0
\end{align*}
and so we obtain the square of the zeros of $\frac{1}{2} J_{\mu_n}(l\lambda) + l\lambda J'_{\mu_n}(l\lambda) = 0 $,
$-\frac{1}{2} J_{\mu_n}(l\lambda) + l\lambda J'_{\mu_n}(l\lambda) = 0 $, $J_{\mu_n}(l\lambda)=0$,
$\lambda^{-\frac{1}{2}}J_{-\frac{1}{2}}(l\lambda)=0$ and $-\lambda^{-\frac{1}{2}}J_{\frac{3}{2}}(l\lambda)=0$, that are
\begin{equation*}\label{ev:abs1}
 \lambda^2 = \frac{\tilde{j}^2_{\mu_n,k,+}}{l^2}, \quad \lambda^2 = \frac{\tilde{j}^2_{\mu_n,k,-}}{l^2}, \quad \lambda^2=\frac{j^2_{\mu_n,k}}{l^2}
 \quad  \mbox{and} \quad \lambda^2=\frac{j^2_{\frac{3}{2},k}}{l^2}.\end{equation*}

Using the relative BC in \eqref{absrel1S2} we have, for the five types,
\begin{align*}
 (\alpha^{(1)}_{n})_{\theta_s}(l,\theta_1,\theta_2) &= (x^{\frac{1}{2}} J_{\mu_n}(\lambda x))(l) = 0\\
 (\beta^{(1)}_{n})_{\theta_s}(l,\theta_1,\theta_2) &= J_{\mu_n}(l\lambda) = 0 \\
 \partial_x(x^2(\beta^{(1)}_{n})_{x})(l,\theta_1,\theta_2) &=
    -\frac{1}{4}J_{\mu_n}(l\lambda) + l\lambda J'_{\mu_n}(l\lambda) + (l\lambda)^{2} J''_{\mu_n}(l\lambda)=0 \\
 \partial_x((\gamma^{(1)}_{n})_{\theta_s})(l,\theta_1,\theta_2) &=
    \frac{1}{2} J_{\mu_n}(l\lambda) + l\lambda J'_{\mu_n}(l\lambda) = 0\\
 \partial_x(x^{2}(\gamma^{(1)}_{n})_{x})(l,\theta_1,\theta_2) &=
    \frac{1}{2} J_{\mu_n}(l\lambda) + l\lambda J'_{\mu_n}(l\lambda) = 0\\
 \partial_x(x^2(D^(1)_n)_x)(l,\theta_1,\theta_2) &= \partial_x(x^2\partial_x(x^{-\frac{1}{2}}J_{\frac{1}{2}}(\lambda x)))(1)=0
\end{align*}
and so we need the square of the zeros of $J_{\mu_n}(l\lambda)=0$, $-\frac{1}{4}J_{\mu_n}(l\lambda) + l\lambda
J'_{\mu_n}(l\lambda) + (l\lambda)^{2} J''_{\mu_n}(l\lambda)=0$, $\frac{1}{2} J_{\mu_n}(l\lambda) + l\lambda
J'_{\mu_n}(l\lambda) = 0$ and $J_{\frac{1}{2}}(l\lambda)=0$, that are
\begin{equation*}\label{ev:rel1} \lambda^2=\frac{j^2_{\mu_n,k}}{l^2} \;\mbox{(twice)},
\quad
 \lambda^2 = \frac{\tilde{j}^2_{\mu_n,k,+}}{l^2} \quad \mbox{and} \quad  \lambda^2 = \frac{j^2_{\frac{1}{2},k}}{l^2}. \end{equation*}

This concludes the proof for 0-forms and 1-forms. The result for 2-forms and 3-forms follows by duality.

\end{proof}

\subsection{Zeta determinant for some simple sequences}
\label{s4.b}

We recall in this section a result of \cite{Spr1}, that will be necessary in the following. For positive real numbers $l$ and $q$, define the non homogeneous quadratic Bessel zeta function by
\[
z(s,\nu,q,l)=\sum_{k=1}^\infty \left(\frac{j_{\nu,k}^2}{l^2}+q^2\right)^{-s},
\]
for $\Re(s)>\frac{1}{2}$. Then, $z(s,\nu,q,l)$ extends analytically to a meromorphic function in the complex plane with simple poles at $s=\frac{1}{2}, -\frac{1}{2}, -\frac{3}{2}, \dots$. The point $s=0$ is a regular point and
\beq\label{p00}
\begin{aligned}
z(0,\nu,q,l)&=-\frac{1}{2}\left(\nu+\frac{1}{2}\right),\\
z'(0,\nu,q,l)&=-\log\sqrt{2\pi l}\frac{I_\nu(lq)}{q^\nu}.
\end{aligned}
\eeq

In particular, taking the limit for $q\to 0$,
\[
z'(0,\nu,0,l)=-\log\frac{\sqrt{\pi}l^{\nu+\frac{1}{2}}}{2^{\nu-\frac{1}{2}}\Gamma(\nu+1)}.
\]

\subsection{Zeta determinant for double sequences of spectral type}
\label{s4.a}

We give in this section all the tools necessary in order to evaluate the zeta determinants appearing in the calculation of the analytic torsion. This is based on \cite{Spr3} \cite{Spr4} \cite{Spr5} and \cite{Spr9}. However, we present here a simplified version of the main result of those works (see in particular the more general formulation in Theorem 3.9 of \cite{Spr9} or the Spectral Decomposition Lemma of \cite{Spr5}), that is sufficient for our purpose here.

Let $S=\{a_n\}_{n=1}^\infty$ be a sequence of non vanishing complex numbers, ordered by increasing modules, with the unique point of accumulation at infinite. The positive real number (possibly infinite)
\[
s_0={\rm limsup}_{n\to\infty} \frac{\log n}{\log |a_n|},
\]
is called the exponent of convergence of $S$, and denoted by $\ec(S)$. We are only interested in sequences with  $\ec(S)=s_0<\infty$. If this is the case, then there exists a least integer $p$ such that the series $\sum_{n=1}^\infty a_n^{-p-1}$ converges absolutely. We assume $s_0-1< p\leq s_0$, we call the integer $p$ the genus of the sequence $S$, and we write $p=\ge(S)$.  We define the zeta function associated to $S$ by the uniformly convergent series
\[
\zeta(s,S)=\sum_{n=1}^\infty a_n^{-s},
\]
when $\Re(s)> \ec(S)$, and by analytic continuation otherwise.  We call the open subset $\rho(S)=\C-S$ of the complex plane the  resolvent set of $S$. For all $\lambda\in\rho(S)$, we define the Gamma function associated to $S$  by the canonical product
\beq\label{gamma}
\frac{1}{\Gamma(-\lambda,S)}=\prod_{n=1}^\infty\left(1+\frac{-\lambda}{a_n}\right)\e^{\sum_{j=1}^{\ge(S)}\frac{(-1)^j}{j}\frac{(-\lambda)^j}{a_n^j}}.
\eeq

When necessary in order to define the meromorphic branch of an analytic function, the domain for $\lambda$ will be the  open subset $\C-[0,\infty)$ of the complex plane.
We use the notation $\Sigma_{\theta,c}=\left\{z\in \C~|~|\arg(z-c)|\leq \frac{\theta}{2}\right\}$,
with $c\geq \delta> 0$, $0< \theta<\pi$. We use
$D_{\theta,c}=\C-\Sigma_{\theta,c}$, for the complementary (open) domain and $\Lambda_{\theta,c}=\partial \Sigma_{\theta,c}=\left\{z\in \C~|~|\arg(z-c)|= \frac{\theta}{2}\right\}$, oriented counter clockwise, for the boundary.
With this notation, we define now a particular subclass of sequences. Let $S$ be as above, and assume that $\ec(S)<\infty$, and that there exist $c>0$ and $0<\theta<\pi$, such that $S$ is contained in the interior of the sector $\Sigma_{\theta,c}$. Furthermore, assume that the logarithm of the associated Gamma function has a uniform asymptotic expansion for large $\lambda\in D_{\theta,c}(S)=\C-\Sigma_{\theta,c}$ of the following form
\[
\log\Gamma(-\lambda,S)\sim\sum_{j=0}^\infty a_{\alpha_j,0}(-\lambda)^{\alpha_j} +\sum_{k=0}^{\ge(S)} a_{k,1}(-\lambda)^k\log(-\lambda),
\]
where $\{\alpha_j\}$ is a decreasing sequence of real numbers. Then, we say that $S$ is a {\it totally regular sequence of spectral type with infinite order}. We call the open set $D_{\theta,c}(S)$ the asymptotic domain of $S$.

Next, let $S=\{\lambda_{n,k}\}_{n,k=1}^\infty$ be a double sequence of non
vanishing complex numbers with unique accumulation point at the
infinity, finite exponent $s_0=\ec(S)$ and genus $p=\ge(S)$. Assume if necessary that the elements of $S$ are ordered as $0<|\lambda_{1,1}|\leq|\lambda_{1,2}|\leq |\lambda_{2,1}|\leq \dots$. We use the notation $S_n$ ($S_k$) to denote the simple sequence with fixed $n$ ($k$). We call the exponents of $S_n$ and $S_k$ the relative exponents of $S$, and we use the notation $(s_0=\ec(S),s_1=\ec(S_k),s_2=\ec(S_n))$. We define relative genus accordingly.

\begin{defi} Let $S=\{\lambda_{n,k}\}_{n,k=1}^\infty$ be a double
sequence with finite exponents $(s_0,s_1,s_2)$, genus
$(p_0,p_1,p_2)$, and positive spectral sector
$\Sigma_{\theta_0,c_0}$. Let $U=\{u_n\}_{n=1}^\infty$ be a totally
regular sequence of spectral type of infinite order with exponent
$r_0$, genus $q$, domain $D_{\phi,d}$. We say that $S$ is
spectrally decomposable over $U$ with power $\kappa$, length $\ell$ and
asymptotic domain $D_{\theta,c}$, with $c={\rm min}(c_0,d,c')$,
$\theta={\rm max}(\theta_0,\phi,\theta')$, if there exist positive
real numbers $\kappa$, $\ell$ (integer), $c'$, and $\theta'$, with
$0< \theta'<\pi$,   such that:
\begin{enumerate}
\item the sequence
$u_n^{-\kappa}S_n=\left\{\frac{\lambda_{n,k}}{u^\kappa_n}\right\}_{k=1}^\infty$ has
spectral sector $\Sigma_{\theta',c'}$, and is a totally regular
sequence of spectral type of infinite order for each $n$;
\item the logarithmic $\Gamma$-function associated to  $S_n/u_n^\kappa$ has an asymptotic expansion  for large
$n$ uniformly in $\lambda$ for $\lambda$ in
$D_{\theta,c}$, of the following form
\beq\label{exp}
\hspace{30pt}\log\Gamma(-\lambda,u_n^{-\kappa} S_n)=\sum_{h=0}^{\ell}
\phi_{\sigma_h}(\lambda) u_n^{-\sigma_h}+\sum_{l=0}^{L}
P_{\rho_l}(\lambda) u_n^{-\rho_l}\log u_n+o(u_n^{-r_0}),
\eeq
where $\sigma_h$ and $\rho_l$ are real numbers with $\sigma_0<\dots <\sigma_\ell$, $\rho_0<\dots <\rho_L$, the
$P_{\rho_l}(\lambda)$ are polynomials in $\lambda$ satisfying the condition $P_{\rho_l}(0)=0$, $\ell$ and $L$ are the larger integers
such that $\sigma_\ell\leq r_0$ and $\rho_L\leq r_0$.
\end{enumerate}
\label{spdec}
\end{defi}

When a double sequence $S$ is spectrally decomposable over a simple sequence $U$, Theorem 3.9 of \cite{Spr9} gives a formula for the derivative of the associated zeta function at zero. In order to understand such a formula, we need to introduce some other quantities. First, we define the functions
\beq\label{fi1}
\Phi_{\sigma_h}(s)=\int_0^\infty t^{s-1}\frac{1}{2\pi i}\int_{\Lambda_{\theta,c}}\frac{\e^{-\lambda t}}{-\lambda} \phi_{\sigma_h}(\lambda) d\lambda dt.
\eeq

Next, by Lemma 3.3 of \cite{Spr9}, for all $n$, we have the expansions:
\beq\label{form}\begin{aligned}
\log\Gamma(-\lambda,S_n/{u_n^\kappa})&\sim\sum_{j=0}^\infty a_{\alpha_j,0,n}
(-\lambda)^{\alpha_j}+\sum_{k=0}^{p_2} a_{k,1,n}(-\lambda)^k\log(-\lambda),\\
\phi_{\sigma_h}(\lambda)&\sim\sum_{j=0}^\infty b_{\sigma_h,\alpha_j,0}
(-\lambda)^{\alpha_j}+\sum_{k=0}^{p_2} b_{\sigma_h,k,1}(-\lambda)^k\log(-\lambda),
\end{aligned}
\eeq
for large $\lambda$ in $D_{\theta,c}$. We set (see Lemma 3.5 of \cite{Spr9})
\beq\label{fi2}
\begin{aligned}
A_{0,0}(s)&=\sum_{n=1}^\infty \left(a_{0, 0,n} -\sum_{h=0}^\ell
b_{\sigma_h,0,0}u_n^{-\sigma_h}\right)u_n^{-\kappa s},\\
A_{j,1}(s)&=\sum_{n=1}^\infty \left(a_{j, 1,n} -\sum_{h=0}^\ell
b_{\sigma_h,j,1}u_n^{-\sigma_h}\right)u_n^{-\kappa s},
~~~0\leq j\leq p_2.
\end{aligned}
\eeq

We can now state the formula for the derivative at zero of the double zeta function. We give here a modified version of Theorem 3.9 of \cite{Spr9}, more suitable for our purpose here. This is based on the following fact. The key point in the proof of Theorem 3.9 of \cite{Spr9} is the decomposition given in Lemma 3.5 of that paper of the sum
\[
\mathcal{T}(s,\lambda, S,U)=\sum_{n=1}^\infty u_n^{-\kappa s} \log\Gamma(-\lambda, u_n^{-\kappa}S_n),
\]
in two terms: the regular part $\mathcal{P}(s,\lambda,S,U)$ and the remaining singular part. The regular part is obtained subtracting from $\T$ some terms constructed starting from the expansion of the logarithmic Gamma function given in equation (\ref{exp}), namely
\[
\P(s,\lambda,S,u)=\T(s,\lambda, S,U)-\sum_{h=0}^{\ell}
\phi_{\sigma_h}(\lambda) u_n^{-\sigma_h}+\sum_{l=0}^{L}
P_{\rho_l}(\lambda).
\]

Now, assume instead we subtract only the terms such that the zeta function $\zeta(s,U)$ has a pole at $s=\sigma_h$ or at $s=\rho_l$. Let $\hat \P(s,\lambda, S,U)$ be the resulting function. Then the same argument as the one used in Section 3 of \cite{Spr9} in order to prove Theorem 3.9 applies, and we obtain similar formulas for the values of the residue, and of the finite part of the zeta function $\zeta(s,S)$ and of its derivative at zero, with just two differences: first, in the all the sums, all the terms with index $\sigma_h$ such that $s=\sigma_h$ is not a pole of $\zeta(s,U)$ must be omitted; and second, we must substitute the terms $A_{0,0}(0)$ and $A_{0,1}'(0)$, with the finite parts $\Rz_{s=0}A_{0,0}(s)$, and $\Rz_{s=0}A_{0,1}'(s)$. The first modification is an obvious consequence of the substitution of the function $\P$ by the function $\hat \P$. The second modification, follows by the same reason noting that the function $A_{\alpha_j,k}(s)$ defined in Lemma 3.5 of \cite{Spr9}, are no longer regular at $s=0$ themselves. However, they both admits a meromorphic extension regular at $s=0$, using the extension of the zeta function $\zeta(s,U)$, and the expansion of the coefficients $a_{\alpha_j,k,n}$ for large $n$.
Thus we have the following result.

\begin{theo} \label{tt} The formulas of Theorem 3.9 of \cite{Spr9} hold if all the quantities with index $\sigma_h$ such that the zeta function $\zeta(s,U)$ has not a pole at $s=\sigma_h$ are omitted.
\end{theo}

Next, assuming some simplified pole structure for the zeta function $\zeta(s,U)$, sufficient for the present analysis, we state the main result of this section.

\begin{theo} \label{t4} Let $S$ be spectrally decomposable over $U$ as in Definition \ref{spdec}. Assume that $\zeta(s,U)$ has only one simple pole at $s=s_0$, and that the function $\Phi_{s_0}(s)$ has at most simple poles. Then,
$\zeta(s,S)$ is regular at $s=0$, and
\begin{align*}
\zeta(0,S)=&-A_{0,1}(0)+\frac{1}{\kappa}\Ru_{s=0}\Phi_{s_0}(s)\Ru_{s=s_0}\zeta(s,U),\\
\zeta'(0,S)=&-A_{0,0}(0)-A_{0,1}'(0)+\frac{\gamma}{\kappa}\Ru_{s=0}\Phi_{s_0}(s)\Ru_{s=s_0}\zeta(s,U)\\
&+\frac{1}{\kappa}\Rz_{s=0}\Phi_{s_0}(s)\Ru_{s=s_0}\zeta(s,U)+\Ru_{s=0}\Phi_{s_0}(s)\Rz_{s=s_0}\zeta(s,U).
\end{align*}

\end{theo}

This result should be compared with the Spectral Decomposition Lemma  of \cite{Spr5} and Proposition 1 of \cite{Spr6}.
The result of Theorem \ref{t4} extends to large class of sequences, where not necessary all the assumptions introduced in the definition of spectral decomposability hold. In particular, we need the following extension.

\begin{corol} Let $S_{(j)}=\{\lambda_{(j),n,k}\}_{n,k=1}^\infty$, $j=1,2$, be two double sequences that satisfy all the requirements of Definition \ref{spdec} of spectral decomposability over a common sequence $U$, with the same parameters $\kappa$, $\ell$, etc., except that the polynomials $P_{(j),\rho}(\lambda)$ appearing in condition (2) do not vanish for $\lambda=0$. Assume that the difference of such polynomials does satisfy this condition, namely that $P_{(1),\rho}(0)-P_{(2),\rho}(0)=0$. Then, the difference of the zeta functions $\zeta(s,S_{(1)})-\zeta(s,S_{(2)})$ is regular at $s=0$ and satisfies the formulas given in Theorem \ref{t4}.
\end{corol}

\subsection{The analytic torsion of a cone over the circle}
\label{s4.2}

It is easy to see that absolute and relative analytic torsions coincide up to sign in this case. Thus we consider absolute boundary conditions. By the analysis in Section \ref{s4.1}, the relevant zeta functions are
\begin{align*}
\zeta(s,\Delta^{(1)})&=\sum_{k=1}^\infty \frac{j_{0,k}^{-2s}}{l^{-2s}}
+\sum_{k=1}^\infty \frac{j_{1,k}^{-2s}}{l^{-2s}}+2\sum_{n,k=1}^\infty \frac{j_{\nu n,k}^{-2s}}{l^{-2s}}
+2\sum_{n,k=1}^\infty \frac{(j_{\nu n,k}')^{-2s}}{l^{-2s}},\\
\zeta(s,\Delta^{(2)})&=\sum_{k=1}^\infty \frac{j_{0,k}^{-2s}}{l^{-2s}}+2\sum_{n,k=1}^\infty
\frac{j_{\nu n,k}^{-2s}}{l^{-2s}},
\end{align*}
and by equation (\ref{analytic}), the torsion is ($a=\sin\alpha=\frac{1}{\nu}$)
\begin{align*}
\log T_{\rm abs}((C_\alpha S^1_{la},g_E);\rho)&=-\frac{1}{2}\zeta'(0,\Delta^{(1)}) +\zeta'(0,\Delta^{(2)}).
\end{align*}

Define the function
\begin{align*}
t(s)&=-\frac{1}{2}\zeta(s,\Delta^{(1)}) +\zeta(s,\Delta^{(2)})\\
&=\frac{1}{2}l^{2s}\sum_{k=1}^\infty j_{0,k}^{-2s}-\frac{1}{2}l^{2s}\sum_{k=1}^\infty j_{1,k}^{-2s}
+l^{2s}\sum_{n,k=1}^\infty j_{\nu n,k}^{-2s}-l^{2s}\sum_{n,k=1}^\infty (j_{\nu n,k}')^{-2s}\\
&=l^{2s}\left(\frac{1}{2}z_0(s)-\frac{1}{2}z_1(s)+Z(s)-\hat Z(s)\right),
\end{align*}
then
\begin{align*}
\log T_{\rm abs}((C_\alpha S^1_{la},g_E);\rho)=t'(0)=&\log l^2\left(\frac{1}{2}z_0(0)-\frac{1}{2}z_1(0)+Z(0)-\hat Z(0)\right)\\
&\hspace{20pt}+\frac{1}{2}z'_0(0)-\frac{1}{2}z'_1(0)+Z'(0)-\hat Z'(0).
\end{align*}

Using equations (\ref{p00})  of Section \ref{s4.a}, we compute $z_{0/1}(0)$ and $z'_{0/1}(0)$. We obtain
\beq\label{ttt}
\log T_{\rm abs}((C_\alpha S^1_{la},g_E);\rho)=\left(\frac{1}{4}+Z(0)-\hat Z(0)\right)\log l^2+Z'(0)-\hat Z'(0)-\frac{1}{2}\log 2.
\eeq

It remains to deal with the differences $Z(0)-\hat Z(0)$ and $Z'(0)-\hat Z'(0)$.
For we use Theorem \ref{t4}. The relevant sequences are the double sequences $S=\{j_{\nu n,k}^2\}$ or $\hat S=\{(j_{\nu n,k}')^2\}$, and the simple sequence $U=\{\nu n\}_{n=1}^\infty$, and $Z(s)=\zeta(s,S)$,
$\hat Z(s)=\zeta(s,\hat S)$. Using classical estimates for the zeros of Bessel function \cite{Wat}, we find  that the genus of $S$ is $1$, the genus of $U$ is $1$, and the relative genus of $S$ are $(1,0,0)$. Second, we check that $U$, $S_n$,  and $\hat S_n$ are totally regular sequences of spectral type. For $U$, this is trivial (see also \cite{Spr4} Section 3.1), since $\zeta(s,U)=\nu^{-s}\zeta(s)$, where $\zeta(s)$ is the Riemann zeta function. For $S$ and $\hat S$, we generalize the argument used in  \cite{Spr3} for the sequence $S$, as in \cite{Spr5} and \cite{Spr6}. Note that we have the following product representations (the first is classical, see for example \cite{Wat}, the second follows using the Hadamard factorization theorem)
\begin{align*}
I_\nu(z)=\frac{z^\nu}{2^\nu\Gamma(\nu+1)}\prod_{k=1}^\infty \left(1+\frac{z^2}{j_{\nu,k}^2}\right),\qquad
I_\nu'(z)=\frac{z^{\nu-1}}{2^\nu\Gamma(\nu)}\prod_{k=1}^\infty \left(1+\frac{z^2}{(j_{\nu,k}')^2}\right).\\
\end{align*}

Using these representations, we obtain the following representations for the Gamma functions associated to the sequences $S_n$ and $\hat S_n$. For further use, we give instead the representations for the Gamma functions associated to the sequences $S_n/u^2_n$, and $S'_n/u^2_n$. By the definition in equation (\ref{gamma}), with $z=\sqrt{-\lambda}$, we have
\begin{align*}
\log \Gamma(-\lambda,S_n/(\nu n)^2)=&-\log\prod_{k=1}^\infty \left(1+\frac{(-\lambda)(\nu n)^2}{j_{\nu n,k}^2}\right)\\
=&-\log I_{\nu n}(\nu n\sqrt{-\lambda})+(\nu n)\log\sqrt{-\lambda} \\
&+\nu n\log (\nu n)-\nu n\log 2-\log\Gamma(\nu n+1),\\
\log \Gamma(-\lambda,\hat S_n/(\nu n)^2)=&-\log\prod_{k=1}^\infty \left(1+\frac{(-\lambda)(\nu n)^2}{(j_{\nu n,k}')^2}\right)\\
=&-\log I'_{\nu n}(\nu n\sqrt{-\lambda})+(\nu n-1)\log\sqrt{-\lambda} \\
&+\nu n\log (\nu n)-\nu n\log 2-\log\Gamma(\nu n+1).
\end{align*}

A first consequence of this representations is that we have a complete asymptotic expansion of the Gamma functions
$\log \Gamma(-\lambda,S_n)$, and $\log \Gamma(-\lambda,\hat S_n)$, and therefore $S_n$ and $\hat S_n$ are sequences
of spectral type. Considering the expansions, it follows that they are both totally regular sequences of infinite order.

Next, we prove that $S$ and $\hat S$ are spectrally decomposable over $U$ with power $\kappa=2$ and length $\ell=2$, as in Definition \ref{spdec}. We have to show that the functions $\log \Gamma(-\lambda,S_n/u_n^2)$, and $\log \Gamma(-\lambda,\hat S_n/u_n^2)$ have the appropriate uniform expansions for large $n$. This follows using the uniform expansions for the Bessel functions given for example in \cite{Olv} (7.18), and Ex. 7.2
\[
I_{\nu}(\nu z)=\frac{\e^{\nu\sqrt{1+z^2}}\e^{\nu\log\frac{z}{1+\sqrt{1+z^2}}}}{\sqrt{2\pi \nu}(1+z^2)^\frac{1}{4}}\left(1+U_1(z)\frac{1}{\nu}+O(\nu^{-2})\right),
\]
with $U_1(z)=\frac{1}{8\sqrt{1+z^2}}-\frac{5}{24(1+z^2)^\frac{3}{2}}$, and
\[
I_{\nu}'(\nu z)=\frac{(1+z^2)^\frac{1}{4}\e^{\nu\sqrt{1+z^2}}\e^{\nu\log\frac{z}{1+\sqrt{1+z^2}}}}{\sqrt{2\pi \nu}z}\left(1+V_1(z)\frac{1}{\nu}+O(\nu^{-2})\right),
\]
with $V_1(z)=-\frac{3}{8\sqrt{1+z^2}}+\frac{7}{24(1+z^2)^\frac{3}{2}}$.
Using the classical expansion for the logarithm of the Euler Gamma function \cite{GZ} 8.344, we obtain, for large $n$, uniformly in $\lambda$,
\begin{align*}
\log \Gamma(-\lambda,S_n/u_n^2)=&\sum_{h=0}^\infty \phi_{h-1}(\lambda) u_n^{1-h}\\
=&\left(1-\log 2-\sqrt{1-\lambda}+\log(1+\sqrt{1-\lambda})\right)\nu n\\
&+\frac{1}{4}\log(1-\lambda)-\left(U_1(\sqrt{-\lambda})+\frac{1}{12}\right)\frac{1}{\nu n}+O\left(\frac{1}{(\nu n)^2}\right),\\
\log \Gamma(-\lambda,\hat S_n/u_n^2)=&\sum_{h=0}^\infty \hat\phi_{h-1}(\lambda) u_n^{1-h}\\
=&\left(1-\log 2-\sqrt{1-\lambda}+\log(1+\sqrt{1-\lambda})\right)\nu n\\
&-\frac{1}{4}\log(1-\lambda)-\left(V_1(\sqrt{-\lambda})+\frac{1}{12}\right)\frac{1}{\nu n}+O\left(\frac{1}{(\nu n)^2}\right),
\end{align*}
and hence in particular
\begin{align*}
\phi_1(\lambda)&=-\frac{1}{8}\frac{1}{(1-\lambda)^\frac{1}{2}}+\frac{5}{24}\frac{1}{(1-\lambda)^\frac{3}{2}}-\frac{1}{12},\\
\hat\phi_1(\lambda)&=\frac{3}{8}\frac{1}{(1-\lambda)^\frac{1}{2}}-\frac{7}{24}\frac{1}{(1-\lambda)^\frac{3}{2}}-\frac{1}{12}.
\end{align*}

Note that the length $\ell$ of the decomposition is precisely $2$. For the $\ec(U)=1$, and therefore the larger integer such that $h-1=\sigma_h\leq 1$ is $2$, since $\sigma_0=-1$, $\sigma_1=0$, $\sigma_2=1$. However, note that by Theorem \ref{tt}, only the term with $\sigma_h=1$, namely $h=2$, appears in the formula of Theorem \ref{t4}, since the unique pole of $\zeta(s,U)$ is at $s=1$.
We now apply the formula in Theorem \ref{t4}. First, since
\begin{align*}
\Rz_{s=1}\zeta(s,U)&=\frac{1}{\nu}(\gamma+\log\nu),&
\Ru_{s=1}\zeta(s,U)&=\frac{1}{\nu},\\
\end{align*}
it follows that
\beq\label{p}
\begin{aligned}
\zeta(0,S)-\zeta(0,\hat S)=&-A_{0,1}(0)+\hat A_{0,1}(0)+\frac{1}{2\nu}\Ru_{s=0}(\Phi_1(s)-\hat\Phi_1(s)),\\
\zeta'(0,S)-\zeta'(0,\hat S)=&-A_{0,0}(0)-A_{0,1}'(0)+\hat A_{0,0}(0)+\hat A_{0,1}'(0)\\
&\hspace{-60pt}+\frac{1}{2\nu}\Rz_{s=0}(\Phi_1(s)-\hat \Phi_1(s))+\frac{1}{\nu}\left(\frac{3}{2}\gamma+\log\nu\right)\Ru_{s=0}(\Phi_{1}(s)-\hat \Phi_1(s)).\\
\end{aligned}
\eeq

Second, by definition in equation (\ref{fi1}),
\begin{align*}
\Phi_1(s)&=\int_0^\infty t^{s-1}\frac{1}{2\pi i}\int_{\Lambda_{\theta,c}}\frac{\e^{-\lambda t}}{-\lambda} \left(-\frac{1}{8}\frac{1}{(1-\lambda)^\frac{1}{2}}+\frac{5}{24}\frac{1}{(1-\lambda)^\frac{3}{2}}-\frac{1}{12}\right) d\lambda dt,\\
\hat \Phi_1(s)&=\int_0^\infty t^{s-1}\frac{1}{2\pi i}\int_{\Lambda_{\theta,c}}\frac{\e^{-\lambda t}}{-\lambda} \left(\frac{3}{8}\frac{1}{(1-\lambda)^\frac{1}{2}}-\frac{7}{24}\frac{1}{(1-\lambda)^\frac{3}{2}}-\frac{1}{12}\right) d\lambda dt.\\
\end{align*}

These integrals are computed in Appendix \ref{appendixA}. We obtain
\begin{align*}
\Phi_1(s)=\frac{\Gamma\left(s+\frac{1}{2}\right)}{12\sqrt{\pi}s}(1+5s),&\hspace{30pt}
\hat \Phi_1(s)=\frac{\Gamma\left(s+\frac{1}{2}\right)}{12\sqrt{\pi}s}(1-7s),\\
\end{align*}
and hence
\begin{align*}
\Rz_{s=0}\Phi_1(s)&=\frac{5-\gamma}{12}-\frac{1}{6}\log 2,&\Ru_{s=0}\Phi_1(s)&=\frac{1}{12},\\
\Rz_{s=0}\hat \Phi_1(s)&=-\frac{7+\gamma}{12}-\frac{1}{6}\log 2,&\Ru_{s=0}\hat\Phi_1(s)&=\frac{1}{12}.\\
\end{align*}

Using equation (\ref{p}), this gives
\beq\label{p1}\begin{aligned}
Z(0)-\hat Z(0)=&-(A_{0,1}(0)-\hat A_{0,1}(0)),\\
Z'(0)-\hat Z'(0)=&-\left(A_{0,0}(0)+A_{0,1}'(0)-\hat A_{0,0}(0)-\hat A_{0,1}'(0)\right)+\frac{1}{2\nu}.\\
\end{aligned}
\eeq

Third, by equation (\ref{fi2}) and Theorem \ref{tt}, the terms $A_{0,0}(0)$ and $A'_{0,1}(0)$, are
\begin{align*}
A_{0,0}(s)&=\sum_{n=1}^\infty \left(a_{0, 0,n} -b_{1,0,0}u_n^{-1}\right)u_n^{-2 s},\\
A_{0,1}(s)&=\sum_{n=1}^\infty \left(a_{0, 1,n} -b_{1,0,1}u_n^{-1}\right)u_n^{-2 s}.
\end{align*}

Hence, we need the expansion for large $\lambda$ of the functions $\log\Gamma(-\lambda,S_n/u_n^2)$ and  $\phi_{1}(\lambda)$.
Using classical expansions for the Bessel functions and their derivative and the formulas in equation (\ref{form}), we obtain
\begin{align*}
a_{0,0,n}&=\frac{1}{2}\log 2\pi+\left(\nu n+\frac{1}{2}\right)\log\nu n-\nu n\log 2-\log\Gamma(\nu n+1),\\
a_{0,1,n}&=\frac{1}{2}\left(\nu n+\frac{1}{2}\right),\\
b_{1,0,0}&=-\frac{1}{12},\hspace{50pt}b_{1,0,1}=0,
\end{align*}
and
\begin{align*}
\hat a_{0,0,n}&=\frac{1}{2}\log 2\pi+\left(\nu n+\frac{1}{2}\right)\log\nu n-\nu n\log 2-\log\Gamma(\nu n+1),\\
\hat a_{0,1,n}&=\frac{1}{2}\left(\nu n-\frac{1}{2}\right),\\
\hat b_{1,0,0}&=-\frac{1}{12},\hspace{50pt}\hat b_{1,0,1}=0.
\end{align*}

This immediately shows that $A_{0,0}(0)=\hat A_{0,0}(0)$, that
\[
A_{0,1}(s)-\hat A_{0,1}(s)=\frac{1}{2}\sum_{n=1}^\infty u_n^{-2s}=\frac{1}{2}\zeta(2s,U)=\frac{1}{2}\nu^{-2s}\zeta(2s),
\]
and hence that
\begin{align*}
A_{0,1}(0)-\hat A_{0,1}(0)=-\frac{1}{4},\qquad
A'_{0,1}(0)-\hat A'_{0,1}(0)=\frac{1}{2}\log \nu-\frac{1}{2}\log 2\pi.
\end{align*}

Substitution in equation (\ref{p1}) gives
\begin{align*}
Z(0)-\hat Z(0)=&\frac{1}{4},\\
Z'(0)-\hat Z'(0)=&-\frac{1}{2}\log \nu+\frac{1}{2}\log 2\pi+\frac{1}{2\nu}.
\end{align*}

Substitution in equation (\ref{ttt}) gives 
\beq\label{tabs}\begin{aligned}
\log T_{\rm abs}((C_\alpha S^1_{l\sin\alpha}),g_E);\rho)=&\frac{1}{2}\log\frac{\pi}{\nu}l^2+\frac{1}{2\nu}.
\end{aligned}
\eeq

\subsection{The analytic torsion of a cone over the sphere}
\label{s4.3}

It is easy to see that absolute and relative analytic torsion coincides in this case. We consider absolute boundary conditions. By the analysis in Section \ref{s4.1}, the relevant zeta functions are
\begin{align*}
\zeta(s,\Delta^{(1)})&=\sum_{k=1}^\infty \frac{j_{\frac{3}{2},k}^{-2s}}{l^{-2s}}+\sum_{n,k=1}^\infty (2n+1)
\frac{j_{\mu_n,k}^{-2s}}{l^{-2s}}
+\sum_{n,k=1}^\infty (2n+1) \frac{\tilde j_{\mu_n ,k,\pm}^{-2s}}{l^{-2s}},\\
\zeta(s,\Delta^{(2)})&=\sum_{k=1}^\infty \frac{j_{\frac{1}{2},k}^{-2s}}{l^{-2s}}+ 2 \sum_{n,k=1}^\infty (2n+1)
\frac{j_{\mu_n,k}^{-2s}}{l^{-2s}}
+\sum_{n,k=1}^\infty (2n+1) \frac{\tilde j_{\mu_n ,k,+}^{-2s}}{l^{-2s}}, \\
\zeta(s,\Delta^{(3)})&=\sum_{k=1}^\infty \frac{j_{\frac{1}{2},k}^{-2s}}{l^{-2s}}+\sum_{n,k=1}^\infty (2n+1)
\frac{j_{\mu_n,k}^{-2s}}{l^{-2s}},
\end{align*}
and by equation (\ref{analytic}), the torsion is
\begin{align*}
\log T_{\rm abs}((C_\alpha S^2_{l\alpha},g_E),\rho)&=-\frac{1}{2}\zeta'(0,\Delta^{(1)}) +\zeta'(0,\Delta^{(2)}) -\frac{3}{2}\zeta'(0,\Delta^{(3)}).
\end{align*}

Define the function
\begin{align*}
t(s)&=-\frac{1}{2}\zeta(s,\Delta^{(1)}) +\zeta(s,\Delta^{(2)})-\frac{3}{2}\zeta(s,\Delta^{(3)})\\
&=-\frac{1}{2}\sum_{k=1}^\infty \frac{j_{\frac{1}{2},k}^{-2s}}{l^{-2s}} - \frac{1}{2}\sum_{k=1}^\infty\hspace{-1pt}
\frac{j_{\frac{3}{2},k}^{-2s}}{l^{-2s}}
+\frac{1}{2}\hspace{-2.3pt}\sum_{n,k=1}^\infty (2n+1) \frac{\tilde j_{\mu_n,k,+}^{-2s}}{l^{-2s}}-\frac{1}{2}\hspace{-2.3pt} \sum_{n,k=1}^\infty\hspace{-1pt}
(2n+1) \frac{\tilde j_{\mu_n,k,-}^{-2s}}{l^{-2s}}\\
&=l^{2s}\left(-\frac{1}{2}z_{\frac{1}{2}}(s)-\frac{1}{2}z_{\frac{3}{2}}(s)+\frac{1}{2} Z_+(s)-\frac{1}{2} Z_-(s)\right),
\end{align*}
then
\begin{align*}
\log T_{\rm abs}((C_\alpha S^2_{l\alpha},g_E),\rho)=t'(0)=&\log l^2\left(-\frac{1}{2} z_{\frac{1}{2}}(0)-
\frac{1}{2}z_{\frac{3}{2}}(0)+\frac{1}{2} Z_+(0)-\frac{1}{2} Z_-(0)\right)\\
&\hspace{20pt}-\frac{1}{2}z'_{\frac{1}{2}}(0)-\frac{1}{2}z'_{\frac{3}{2}}(0)+\frac{1}{2} Z'_+(0)-\frac{1}{2} Z'_-(0).
\end{align*}

Using equations (\ref{p00}) of Section \ref{s4.a}, we compute
\beq\label{ttt2}
\log T_{\rm abs}= \left(\frac{3}{4} + \frac{1}{2} Z_+(0) - \frac{1}{2} Z_-(0) \right) \log l^{2} +
\frac{1}{2}Z_+'(0)-\frac{1}{2}Z_-'(0)+\frac{1}{2}\log\frac{4}{3}
\eeq

It remains to deal with the differences $Z_+(0)- Z_-(0)$ and $Z_+'(0)- Z_-'(0)$.
For we use Theorem \ref{t4}, in the form given in the corollary. The relevant sequences are the double sequences $S_\pm=\{\tilde j^2_{\mu_n,k,\pm}\}$, and the simple sequence $U=\{(2n+1):\mu_n\}_{n=1}^\infty$,
where
\[
u_n=\mu_n=\sqrt{\nu^2 n(n+1)+\frac{1}{4}},
\]
and $Z_\pm(s)=\zeta(s,S_\pm)$. Using classical estimates for the zeros of Bessel function \cite{Wat},  the genus of $S_\pm$ is $0$, the genus of $U$ is $2$, and the relative genus of $S$ are $(1,0,1)$. This only differs from the case of the circle by $\ge(S_{\pm,k})$, with fixed $k$. Using classical estimates for the zeros of the Bessel function, the behavior of this sequence is given by the behavior of the sequence of the eigenvalues of the Laplacian on the sphere $S^2$, that is known. In particular, we recall the main features here below.
We check that $U$, and $S_{\pm,n}$ are totally regular sequences of spectral type. By definition of the sequence
$U$,  $\zeta(s,U)=\nu^{-s} \zeta(s,L_{\frac{1}{4\nu^2}})$, where $L_q=\{(2n+1):\sqrt{n(n+1)+q}\}_{n=1}^\infty$. Hence,  $U$ is related to the sequence of the eigenvalues of the Laplacian on the 2 sphere shifted by some positive constant $q$. More precisely,
$\zeta(2s,L_0)=\zeta(s,\Sp_+ \Delta^{(0)}_{S^2})$.
The zeta function $\zeta(s,\Sp_+ \Delta^{(0)}_{S^2})$ has been studied in \cite{Spr4}, Section 3.3, where it was proved that $\ec(\Sp_+\Delta^{(0)}_{S^2})=\ge(\Sp_+ \Delta^{(0)}_{S^2})=1$, and that $\Sp_+ \Delta^{(0)}_{S^2}$ is a totally regular sequence of spectral type with infinite order, by giving the explicit formula for the associated Gamma function $\Gamma(-\lambda,U)$ in terms of the Barnes $G$ function. It follows that $\ec(U)=\ge(U)=2$, and that $U$ is a totally regular sequence of spectral type with infinite order. Also, $\zeta(s,\Sp_+ \Delta^{(0)}_{S^2})$ has one simple pole at $s=1$, with residues
\begin{align*}
&\Rz_{s=1}\zeta(s,\Sp_+ \Delta^{(0)}_{S^2})=2\gamma,&\Ru_{s=1}\zeta(s,\Sp_+ \Delta^{(0)}_{S^2})=1,\\
\end{align*}
and hence, $\zeta(s,L_0)$ has one simple pole at $s=2$, with the same residues. Expanding the power of the binomial, we have that
\begin{align*}
\zeta(s,L_q)&=\zeta(s,L_0)+f(s),
\end{align*}
where $f(s)$ is regular at $s=2$. Therefore,
\begin{align*}
&\Rz_{s=2}\zeta(s,L_q)=2\gamma+f(2),&\Ru_{s=2}\zeta(s,L_q)=2,\\
\end{align*}
and
\begin{align*}
\zeta(s,U)=\nu^{-s}\zeta(s,L_q)=\frac{2}{\nu^2}\frac{1}{s-2}+K,
\end{align*}
near $s=2$, where $K$ is some finite constant.
For $S_\pm$, we proceed as follows. Let define the function
\[
G^\pm_\nu(z)=\pm\frac{1}{2}J_\nu(z)+zJ'_\nu(z).
\]

Recalling the series definition of the Bessel function
\[
J_\nu(z)=\frac{z^\nu}{2^\nu}\sum_{k=0}^\infty \frac{(-1)^kz^{2k}}{2^{2k}k!\Gamma(\nu+k+1)},
\]
we obtain that near $z=0$
\[
G_\nu^\pm(z) =\left(1\pm\frac{1}{2\nu}\right) \frac{z^\nu}{2^\nu\Gamma(\nu)}.
\]

This means that the function $\hat G^\pm_\nu(z)=z^{-\nu} G^\pm_\nu(z)$ is an even function of $z$.
Let $z_{\nu,k,\pm}$ be the positive zeros of $G^\pm_\nu(z)$ arranged in increasing order. By the Hadamard factorization Theorem, we have the product expansion
\[
\hat G^\pm_\nu(z)=\hat G^\pm_\nu(z){\prod_{k=-\infty}^{+\infty}}'\left(1-\frac{z}{z_{\nu,k,\pm}}\right),
\]
and therefore
\[
G^\pm_\nu(z)=\left(1\pm\frac{1}{2\nu}\right)\frac{z^\nu}{2^\nu\Gamma(\nu)}
\prod_{k=1}^{\infty}\left(1-\frac{z^2}{z^2_{\nu,k,\pm}}\right).
\]

Next,  recalling that (when $-\pi<\arg(z)<\frac{\pi}{2}$)
\begin{align*}
J_\nu(iz)&=\e^{\frac{\pi}{2}i\nu} I_\nu(z),\\
J'_\nu(iz)&=\e^{\frac{\pi}{2}i\nu}\e^{-\frac{\pi}{2}i} I'_\nu(z),\\
\end{align*}
we obtain
\[
G_\nu^\pm(iz)=\e^{\frac{\pi}{2}i\nu}\left(\pm\frac{1}{2}I_\nu(z)+zI'_\nu(z)\right).
\]

Thus, we define (for $-\pi<\arg(z)<\frac{\pi}{2}$)
\beq\label{pop}
H^\pm_\nu(z)=\e^{-\frac{\pi}{2}i\nu}G_\nu^\pm(i z),
\eeq
and hence
\begin{align*}
H^\pm_\nu(z)&=\pm\frac{1}{2}I_\nu(z)+zI'_\nu(z)=\left(1\pm\frac{1}{2\nu}\right)\frac{z^\nu}{2^\nu\Gamma(\nu)}
\prod_{k=1}^{\infty}\left(1+\frac{z^2}{z^2_{\nu,k,\pm}}\right).
\end{align*}

Using these representations, we obtain the following representations for the Gamma functions associated to the sequences $S_{\pm,n}$. By the definition in equation (\ref{gamma}), with $z=\sqrt{-\lambda}$, we have
\begin{align*}
\log \Gamma(-\lambda,S_{\pm,n})=&-\log\prod_{k=1}^\infty \left(1+\frac{(-\lambda)}{\tilde j_{\mu_n,k,\pm}^2}\right)\\
=&-\log H^\pm_{\mu_n}(\sqrt{-\lambda})+\mu_n\log\sqrt{-\lambda} -\log2^{\mu_n}\Gamma(\mu_n)+\log\left(1\pm\frac{1}{2\mu_n}\right)\hspace{-4pt}.
\end{align*}

A first consequence of this representations is that we have a complete asymptotic expansion of the Gamma functions
$\log \Gamma(-\lambda,S_{\pm,n})$, and therefore $S_n$ and $\hat S_n$ are sequences
of spectral type. Considering the expansions, it follows that they are both totally regular sequences of infinite order.

Next, we prove that $S_\pm$ are spectrally decomposable over $U$ with power $\kappa=2$ and length $\ell=3$, as in Definition \ref{spdec}. We have to show that the functions $\log \Gamma(-\lambda,S_{\pm,n}/u_n^2)$, have the appropriate uniform expansions for large $n$. We have
\begin{align*}
\log \Gamma(-\lambda,S_{\pm,n}/\mu^2_n)=&-\log H^\pm_{\mu_n}(\mu_n\sqrt{-\lambda})+\mu_n\log\sqrt{-\lambda}+\mu_n\log\mu_n\\ &-\mu_n\log2 -\log\Gamma(\mu_n)+\log\left(1\pm\frac{1}{2\mu_n}\right).
\end{align*}

Recalling the expansions given in Section \ref{s4.2}, we obtain
\begin{align*}
H^\pm_\nu(\nu z)
&=\sqrt{\nu}(1+z^2)^\frac{1}{4}\frac{\e^{\nu\sqrt{1+z^2}}\e^{\nu\log\frac{z}{1+\sqrt{1+z^2}}}}{\sqrt{2\pi }}\\
&\hspace{30pt}\left(1+W_{1,\pm}(z)\frac{1}{\nu}+W_{2,\pm}(z)\frac{1}{\nu^2}+O(\nu^{-3})\right),
\end{align*}
where $p=\frac{1}{(1-\lambda)^\frac{1}{2}}$, and
\begin{align*}
&W_{1,\pm}(p)=V_1(p)\pm\frac{1}{2}p,&W_{2,\pm}(p)=V_2(p)\pm \frac{1}{2}pU_1(p),
\end{align*}
\begin{align*}
&W_{1,+}(p)=\frac{1}{8}p+\frac{7}{24}p^3,&W_{2,+}(p)=-\frac{7}{128}p^2+\frac{79}{192}p^4-\frac{455}{1152}p^6,\\
&W_{1,-}(p)=-\frac{7}{8}p+\frac{7}{24}p^3,&W_{2,-}(p)=-\frac{28}{128}p^2+\frac{119}{192}p^4-\frac{455}{1152}p^6.\\
\end{align*}

This gives,
\begin{align*}
\log\Gamma(-\lambda, S_{n,\pm}/\mu^2_n)=&\sum_{h=0}^\infty \phi_{h-1,\pm}(\lambda) \mu_n^{1-h}\\
=&\left(1-\sqrt{1-\lambda}+\log(1+\sqrt{1-\lambda})-\log 2\right)\mu_n\\
&-\frac{1}{4}\log(1-\lambda)+\left(-W_{1,\pm}(\sqrt{-\lambda})\pm \frac{1}{2}-\frac{1}{12}\right)\frac{1}{\mu_n}\\
&+\left(-W_{2,\pm}(\sqrt{-\lambda})+\frac{1}{2}W_{1,\pm}^2(\sqrt{-\lambda})-\frac{1}{8}\right)\frac{1}{\mu^2_n}+O\left(\frac{1}{\mu_n^3}\right),
\end{align*}
and hence
\begin{align*}
\phi_{1,+}(\lambda)&=-\frac{1}{8}\frac{1}{(1-\lambda)^\frac{1}{2}}-\frac{7}{24}\frac{1}{(1-\lambda)^\frac{3}{2}}+\frac{5}{12},\\
\phi_{1,-}(\lambda)&=\frac{7}{8}\frac{1}{(1-\lambda)^\frac{1}{2}}-\frac{7}{24}\frac{1}{(1-\lambda)^\frac{3}{2}}-\frac{7}{12},
\end{align*}
\beq\label{pu}
\begin{aligned}
\phi_{2,+}(\lambda)&=\frac{1}{16}\frac{1}{1-\lambda}-\frac{3}{8}\frac{1}{(1-\lambda)^2}+\frac{7}{16}\frac{1}{(1-\lambda)^3}-\frac{1}{8},\\
\phi_{2,-}(\lambda)&=\frac{9}{16}\frac{1}{1-\lambda}-\frac{7}{8}\frac{1}{(1-\lambda)^2}+\frac{7}{16}\frac{1}{(1-\lambda)^3}-\frac{1}{8}.\\
\end{aligned}
\eeq

Note that the length $\ell$ of the decomposition is precisely $3$. For the $\ec(U)=2$, and therefore the larger integer such that $h-1=\sigma_h\leq 2$ is $3$, since $\sigma_0=-1$, $\sigma_1=0$, $\sigma_2=1$, $\sigma_3=2$. However, note that by Theorem \ref{tt}, only the term with $\sigma_h=2$, namely $h=3$, appears in the formula of Theorem \ref{t4}, since the unique pole of $\zeta(s,U)$ is at $s=2$.

We now apply the formula in Theorem \ref{t4}. First, since $\kappa=2$, $\Rz_{s=2}\zeta(s,U)=K$, and
$\Ru_{s=2}\zeta(s,U)=\frac{2}{\nu^2}$, we obtain
\begin{align*}
\zeta(0,S_+) - \zeta(0,S_-)=&-A_{0,1,+}(0)+A_{0,1,-}(0) + \frac{1}{\nu^2} \Ru_{s=0} (\Phi_{2,+}(s) - \Phi_{2,-}(s))\\
\zeta'(0,S_+)-\zeta'(s,S_-)=&-(A_{0,0,+}(0)+A_{0,1,+}'(0)-A_{0,0,-}(0)- A_{0,1,-}'(0))\\
&+\frac{1}{\nu^2}\Rz_{s=0}(\Phi_{2,+}(s)-\Phi_{2,-}(s))\\
&+\left(\frac{\gamma}{\nu^2}+K\right)\Ru_{s=0}(\Phi_{2,+}(s)-\Phi_{2,-}(s)).\\
\end{align*}

Second, by equation (\ref{pu}),
\[
\phi_{2,+}(\lambda)-\phi_{2,-}(\lambda)=-\frac{1}{2}\left(\frac{1}{1-\lambda}-\frac{1}{(1-\lambda)^2}\right),
\]
and hence using the definition in equation (\ref{fi1}),
\begin{align*}
\Phi_{2,+}(s)-\Phi_{2,-}(s)&=-\frac{1}{2}\int_0^\infty t^{s-1}\frac{1}{2\pi i}\int_{\Lambda_{\theta,c}}\frac{\e^{-\lambda t}}{-\lambda}\left(\frac{1}{1-\lambda}-\frac{1}{(1-\lambda)^2}\right).\\
\end{align*}

Using the formula in Appendix \ref{appendixA}, we obtain
\begin{align*}
\Phi_{2,+}(s)-\Phi_{2,-}(s)&=\frac{1}{2}\Gamma(s+1),\\
\end{align*}
and hence
\begin{align*}
&\Rz_{s=0}(\Phi_{2,+}(s)-\Phi_{2,-}(s))=\frac{1}{2},&\Ru_{s=0}(\Phi_{2,+}(s)-\Phi_{2,-}(s))=0.\\
\end{align*}

This gives
\beq\label{pi}
\begin{aligned}
Z_+(0)-Z_-(0)&= -A_{0,1,+}(0) + A_{0,1,-}(0)\\
Z_+'(0)-Z_-'(0)&=\zeta'(0,S_+)-\zeta'(s,S_-)\\
=&-(A_{0,0,+}(0)+A_{0,1,+}'(0)-A_{0,0,-}(0)- A_{0,1,-}'(0))+\frac{1}{2\nu^2}.\\
\end{aligned}
\eeq

Third, by equation (\ref{fi2}) and Theorem \ref{tt}, the terms $A_{0,0}(s)$ and $A_{0,1}(s)$, are
\begin{align*}
A_{0,0,\pm}(s)&=\sum_{n=1}^\infty \left(a_{0, 0,n,\pm} -b_{1,0,0,\pm}u_n^{-1}\right)u_n^{-2 s},\\
A_{0,1,\pm}(s)&=\sum_{n=1}^\infty \left(a_{0, 1,n,\pm} -b_{1,0,1,\pm}u_n^{-1}\right)u_n^{-2 s}.
\end{align*}

Hence, we need the expansion for large $\lambda$ of the functions $\log\Gamma(-\lambda,S_{n,\pm}/u_n^2)$ and  $\phi_{2,\pm}(\lambda)$. Using equations (\ref{pop}) and the definition, we obtain
\[
H^\pm_\nu(z)\sim \frac{\sqrt{z}\e^z}{\sqrt{2\pi}}\left(1+\sum_{k=1}^\infty b_kz^{-k}\right)+O(\e^{-z}),
\]
for large $z$. Therefore,
\begin{align*}
\log\Gamma(-\lambda,S_{n,\pm}/\mu_n^2)=&-\mu_n \sqrt{-\lambda}+\frac{1}{2}\left(\mu_n-\frac{1}{2}\right)\log(-\lambda)
+\frac{1}{2}\log 2\pi\\
&+\left(\mu_n-\frac{1}{2}\right)\log\mu_n
-\log 2^{\mu_n}\Gamma(\mu_n)\\
&+\log\left(1\pm\frac{1}{2\mu_n}\right) +O\left(\frac{1}{\sqrt{-\lambda}}\right).
\end{align*}

Also, from equation (\ref{pu}),
\begin{align*}
\phi_{2,\pm}(\lambda)=-\frac{1}{8}+O\left(\frac{1}{-\lambda}\right).
\end{align*}

Thus,
\begin{align*}
a_{0,0,n,\pm}&=\frac{1}{2}\log 2\pi+\left(\mu_n-\frac{1}{2}\right)\log\mu_n-\log 2^{\mu_n}\Gamma(\mu_n)
+\log\left(1\pm\frac{1}{2\mu_n}\right),\\
a_{0,1,n,\pm}&=\frac{1}{2}\left(\mu_n-\frac{1}{2}\right),\\
b_{2,0,0,\pm}&=-\frac{1}{8},\hspace{30pt}b_{2,0,1,\pm}=0.\\
\end{align*}

This immediately shows that $A_{0,1,+}(s)=A_{0,1,-}(s)$, and therefore $Z_+(0)-Z_-(0) = 0$. Next,
\begin{align*}
A_{0,0,+}(s)-A_{0,0,-}(s)&=\sum_{n=1}^\infty (2n+1) \mu_n^{-2s}\left(\log\left(1+\frac{1}{2\mu_n}\right)-\log\left(1-\frac{1}{2\mu_n}\right)\right)\\
&=F(s,\nu).
\end{align*}

Note that this series converges uniformely for $\Re(s)>2$, but using the analytic extension of the zeta function $\zeta(s,U)$, has an analytic extension that is regular at $s=0$.
Substitution in equation (\ref{pi}) gives
\begin{align*}
Z_+'(0)- Z_-'(0)=&-F(0,\nu)+\frac{1}{2\nu^2}.
\end{align*}

Substitution in equation (\ref{ttt2}) gives
\beq\label{te2}
\begin{aligned}
\log T((C_\alpha S^2_{l\sin\alpha},g_E),\rho))=&\frac{1}{2}\log\frac{4 l^3}{3}-\frac{1}{2}F(0,\csc\alpha)+\frac{1}{4}\sin^2\alpha.
\end{aligned}
\eeq

We give in the Appendix \ref{appendixB} a series representation for the $F(0,\nu)$ for $\nu>1$. Consider here the case
$\nu=1$ deserves independent treatment. Note that, when $\mu=1$, $\mu_n=\sqrt{n(n+1)+\frac{1}{4}}=n+\frac{1}{2}$,
and therefore
\begin{align*}
F(s,1)
=&2^{2s}\sum_{n=1}^\infty (2n+1)^{1-2s}(\log(n+1)-\log n).
\end{align*}

For $\Re(s)>2$, due to absolute convergence, we can rearrange the terms in the sum. We obtain
\begin{align*}
F(s,1)=&-2^{2s}\sum_{n=1}^\infty \left((2n+1)^{1-2s}-(2n-1)^{1-2s}\right)\log n\\
=&\sum_{j=0}^\infty \binom{1-2s}{j}\frac{(1-(-1)^j)}{2^j} \zeta'(2s+j-1)\\
=&(1-2s)\zeta'(2s)+\sum_{k=1}^\infty \binom{1-2s}{2k+1}\frac{\zeta'(2s+2k)}{2^{2k+1}},\\
\end{align*}
and hence, by substitution in equation (\ref{te2}),
\begin{align*}
\log T((C_\frac{\pi}{2} S^2_{l},g_E),\rho)=\log T((D^3_l,g_E),\rho)
=&\frac{1}{2}\log\frac{4\pi l^3}{3}+\frac{1}{2}\log 2+\frac{1}{4}.
\end{align*}

{\bf Acknowledgments}

One of the authors, M.S., thanks W. M\"{u}ller for usefull discussion, remarks and suggestions.

\section{Appendix A}
\label{appendixA}

We give here a formula for a contour integral appearing in the text. The proof is in \cite{Spr3} Section 4.2. Let
$\Lambda_{\theta,c}=\{\lambda\in\C~|~|\arg(\lambda-c)|=\theta\}$,
$0<\theta<\pi$, $0<c<1$, $a$ real, then
\[
\int_0^\infty t^{s-1} \frac{1}{2\pi
i}\int_{\Lambda_{\theta,c}}\frac{\e^{-\lambda
t}}{-\lambda}\frac{1}{(1-\lambda)^a}d\lambda
dt=\frac{\Gamma(s+a)}{\Gamma(a)s}.
\]

\section{Appendix B}
\label{appendixB}

We give a power series representation for the function $F(0,\nu)$ for $\nu>1$. Assume $\Re(s)>2$, then
\begin{align*}
F(s,\nu)=&\sum_{n=1}^\infty (2n+1) \mu_n^{-2s}\left(\log\left(1+\frac{1}{2\mu_n}\right)-\log\left(1-\frac{1}{2\mu_n}\right)\right)\\
=&\sum_{n=1}^\infty (2n+1) \mu_n^{-2s}\sum_{k=0}^\infty \frac{2^{-2k}}{2k+1} \mu^{-2k-1}\\
=&\sum_{k=0}^\infty \frac{1}{(2k+1)2^{2k}}\sum_{n=1}^\infty (2n+1) \mu_n^{-2s-2k-1}.\\
\end{align*}

Now,
\[
\mu_n^{2x}=\left(\nu^2 n(n+1)+\frac{1}{4}\right)^x= \sum_{j=0}^\infty \frac{1}{2^{2j}}\binom{x}{j}(n(n+1))^{x-j}\nu^{2x-2j},
\]
and therefore
\[
F(s,\nu)=\sum_{k=0}^\infty \frac{1}{(2k+1)2^{2k}}\sum_{j=0}^\infty \frac{1}{2^{2j}}\binom{-s-k-\frac{1}{2}}{j}\frac{\zeta(s+k+j+\frac{1}{2},\Sp_+\Delta^{(0)}_{S^2})}{\nu^{2s+2k+2j+1}},
\]
where
\[
\zeta(s,\Sp_+\Delta^{(0)}_{S^2})=\sum_{n=1}^\infty (2n+1)(n(n+1))^{-s}.
\]

Since the unique pole of the meromorphic extension of $\zeta(s,\Sp_+\Delta^{(0)}_{S^2})$ is at $s=1$, writing
\begin{align*}
F(s,\nu)=&\zeta(s,\Sp_+\Delta^{(0)}_{S^2})\nu^{-2s-1}\\
&\hspace{-33pt}+\sum_{\substack{j,k=0,\\ j+k\not=0}}^\infty \frac{1}{(2k+1)2^{2k}} \frac{1}{2^{2j}}\binom{-s-k-\frac{1}{2}}{j}\frac{\zeta(s+k+j+\frac{1}{2},\Sp^{(0)}_+\Delta_{S^2})}{\nu^{2s+2k+2j+1}},
\end{align*}
and using the analytic extension of $\zeta(s,\Sp_+\Delta^{(0)}_{S^2})$, we obtain
\begin{align*}
F(0,\nu)=\zeta(\frac{1}{2},\Sp_+\Delta^{(0)}_{S^2})\frac{1}{\nu}
\hspace{-2pt}+\hspace{-5pt}\sum_{\substack{j,k=0,\\ j+k\not=0}}^\infty \frac{1}{(2k+1)2^{2k}} \frac{1}{2^{2j}}\binom{-k-\frac{1}{2}}{j}\frac{\zeta(k+j+\frac{1}{2},\Sp_+\Delta^{(0)}_{S^2})}{\nu^{2k+2j+1}}.
\end{align*}

It is easy to see that the coefficient in  the power series above are all convergent series, and can be evaluated numerically. The leading  term requires independent treatment. Using the theorem of Plana as in \cite{Spr0}, we obtain
\begin{align*}
\zeta(\frac{1}{2},\Sp_+\Delta^{(0)}_{S^2})=&-\frac{5}{4}\sqrt{2}+6\int_0^\infty \frac{(y^4+y^2+4)^{-\frac{1}{4}}}{\e^{2\pi y}-1}\sin\left(\frac{1}{2}\arctan\frac{3y}{2-y^2}\right) dy\\
&-4\int_0^\infty \frac{(y^4+y^2+4)^{-\frac{1}{4}}}{\e^{2\pi y}-1}\cos\left(\frac{1}{2}\arctan\frac{3y}{2-y^2}\right) dy.
\end{align*}

\end{document}